\documentclass[final,onefignum,onetabnum]{siamart220329}

\usepackage{braket,amsfonts}
\usepackage{essay-defr}

\usepackage{array}

\usepackage[caption=false]{subfig}
\usepackage{pgfplots}
\pgfplotsset{compat=1.16}

\newsiamthm{claim}{Claim}
\newsiamremark{remark}{Remark}
\newsiamremark{hypothesis}{Hypothesis}
\newsiamremark{assumption}{Assumption}
\crefname{hypothesis}{Hypothesis}{Hypotheses}

\usepackage{algorithmic}
\usepackage{booktabs}

\usepackage{graphicx,epstopdf}

\Crefname{ALC@unique}{Line}{Lines}

\usepackage{amsopn}

\usepackage{xspace}
\usepackage{xr}
\usepackage{bold-extra}
\usepackage[most]{tcolorbox}

\colorlet{texcscolor}{blue!50!black}
\colorlet{texemcolor}{red!70!black}
\colorlet{texpreamble}{red!70!black}
\colorlet{codebackground}{black!25!white!25}
\newcommand{\probName}{machine-learning augmented hybrid simulation}
\newcommand{\algName}{tangent-space regularized estimator}
\newcommand{\abbrvalgName}{\mathrm{TR}}
\newcommand{\eig}{\mathrm{eig}}
\newcommand{\OLS}{\mathrm{OLS}}
\lstdefinestyle{siamlatex}{%
  style=tcblatex,
  texcsstyle=*\color{texcscolor},
  texcsstyle=[2]\color{texemcolor},
  keywordstyle=[2]\color{texemcolor},
  moretexcs={cref,Cref,maketitle,mathcal,text,headers,email,url},
}

\tcbset{%
  colframe=black!75!white!75,
  coltitle=white,
  colback=codebackground, % bottom/left side
  colbacklower=white, % top/right side
  fonttitle=\bfseries,
  arc=0pt,outer arc=0pt,
  top=1pt,bottom=1pt,left=1mm,right=1mm,middle=1mm,boxsep=1mm,
  leftrule=0.3mm,rightrule=0.3mm,toprule=0.3mm,bottomrule=0.3mm,
  listing options={style=siamlatex}
}

\newtcblisting[use counter=example]{example}[2][]{%
  title={Example~\thetcbcounter: #2},#1}

\newtcbinputlisting[use counter=example]{\examplefile}[3][]{%
  title={Example~\thetcbcounter: #2},listing file={#3},#1}

\DeclareTotalTCBox{\code}{ v O{} }
{ %fontupper=\ttfamily\color{texemcolor},
  fontupper=\ttfamily\color{black},
  nobeforeafter,
  tcbox raise base,
  colback=codebackground,colframe=white,
  top=0pt,bottom=0pt,left=0mm,right=0mm,
  leftrule=0pt,rightrule=0pt,toprule=0mm,bottomrule=0mm,
  boxsep=0.5mm,
  #2}{#1}

\patchcmd\newpage{\vfil}{}{}{}
\begin{tcbverbatimwrite}{tmp_\jobname_header.tex}
\title{Mitigating distribution shift in machine learning-augmented hybrid simulation}

\author{
    Jiaxi Zhao\thanks{Department of Mathematics, National University of Singapore, 117543, Singapore (\email{jiaxi.zhao@u.nus.edu}).}
    \and
    Qianxiao Li\thanks{Department of Mathematics \&
    Institute for Functional Intelligent Materials, National University of Singapore, 117543, Singapore (\email{qianxiao@nus.edu.sg}).}
}

\headers{Mitigating distribution shift in MLHS}{Zhao and Li}
\end{tcbverbatimwrite}
\input{tmp_\jobname_header.tex}

\ifpdf
\hypersetup{ pdftitle={Guide to Using  SIAM'S \LaTeX\ Style} }
\fi

\begin{document}

\maketitle

%% ------------------------------------------------------------------
%% ABSTRACT
%% ------------------------------------------------------------------
\begin{tcbverbatimwrite}{tmp_\jobname_abstract.tex}
\begin{abstract}

    We study the problem of distribution shift generally arising in \probName,
    where parts of simulation algorithms are
    replaced by data-driven surrogates.
    A mathematical framework is established to 
    understand the structure of \probName \ problems and the cause and effect of the associated distribution shift.
    We show correlations between distribution shift and simulation error both
    numerically and theoretically. Then, we propose a simple methodology based
    on \algName \ to control the distribution shift,
    thereby improving the long-term accuracy of the simulation results.
    In the linear dynamics case, we provide a thorough theoretical analysis to
    quantify the effectiveness of the proposed method. Moreover, we conduct
    several numerical experiments, including simulating a partially known
    reaction-diffusion equation and solving Navier-Stokes equations using the
    projection method with a data-driven pressure solver. In all cases, we
    observe marked improvements in simulation accuracy under the proposed
    method, especially for systems with high degrees of distribution shift, such
    as those with relatively strong non-linear reaction mechanisms,
    or flows at large Reynolds numbers.

\end{abstract}

\begin{keywords}
    machine learning,
    distribution shift,
    regularization,
    error analysis,
    fluid dynamics
\end{keywords}

\begin{MSCcodes}
68T99, 65M15, 37M05
\end{MSCcodes}
\end{tcbverbatimwrite}
\input{tmp_\jobname_abstract.tex}
%% ------------------------------------------------------------------
%% END HEADER
%% ------------------------------------------------------------------

\section{Introduction}
% \indent
Many scientific computational applications, such as computational fluid dynamics (CFD)
and molecular dynamics (MD) can be
viewed as dynamical system modeling and simulation problems, which are tackled by
rigorous numerical tools with theoretical guarantee~\cite{slotnick2014cfd}. 
However, in many cases a part of the simulation workflow, such
as the Reynolds stresses in Reynolds-averaged Navier-Stokes equation (RANS)~\cite{alfonsi2009reynolds} and the exchange-correlation energies
in density function theory used to compute force fields that drive
MD simulations~\cite{lin2019numerical}, depends on models that are
either expensive to compute or even unknown in practice.
One thus often resorts to a \emph{hybrid} simulation method,
where the known, resolved components of the dynamics are computed exactly,
while the unresolved components are replaced by
approximate, but computationally tractable models. For example, solving Navier-Stokes equations using the projection method involves 
two steps. In the first step, all the terms except for the gradient of the pressure are used to evolve the velocity. This step is 
computationally cheap and thus understood as the resolved part. Next, the pressure is solved from a Poisson equation and then used 
to correct the velocity. Most of the computational cost is contained in solving this Poisson equation, and we thereby viewed this 
as the unresolved part. We call such scientific computing problems with both resolved and unresolved
parts ``hybrid simulation problems''. Similar problems are surveyed in~\cite{wang2021physics} under the name of ``Hybrid physics-DL 
models''.

As machine learning becomes increasingly powerful in areas like computer vision
and natural language processing, practitioners begin to use data-driven modules
to model the unresolved part to carry out the simulation.
For example, in~\cite{tompson2017accelerating}
the author replaced the numerical Poisson solver
of the unresolved part with a convolutional neural network trained using a novel
unsupervised learning framework. Then, this data-driven model is coupled with
resolved models and provides fast and realistic simulation results in 2D and 3D.
Similar ideas are used elsewhere, e.g. RANS~\cite{ling2016reynolds} and LES~\cite{shankar2023differentiable2}. We
hereafter refer to this simulation procedure as machine-learning augmented
hybrid simulation (MLHS). This structure represents a great part of the scientific machine learning 
research~\cite{yu2021onsagernet, stachenfeld2022learned} and will be the focus of this paper.
In MLHS, a common problem appears: while the data-drive model performs well on the training data, the performance quickly 
deteriorates when iteratively applying it in simulation, driving the dynamics to some regimes that are not observed from the
training source. Empirical evidences have been observed in various applications,
such as CFD~\cite{shankar2023differentiable2, stachenfeld2022learned, pedersen2023reliable, yu2021onsagernet}, molecular
dynamics~\cite{zhang2019active, zhang2018deep, lin2020data}, and iterative numerical solver~\cite{10.5555/3618408.3618451}, 
but these problems are not
well-studied algorithmically and theoretically in the literature.

This issue has strong connections with the so-called distribution shift (DS) in
computer science applications,
especially in reinforcement learning and computer vision.
Researchers use DS to refer to problems where training and
testing distribution differ significantly. An example is an image classifier
trained with images taken in the daytime that is tested under night conditions~\cite{koh2021wilds}. Therefore, the accuracy during the training itself cannot
guarantee the performance during inference. To resolve this,
researchers have systematically
developed several methods, e.g. domain adaptation~\cite{farahani2021brief},
ensemble learning~\cite{dong2020survey}, and
meta-learning~\cite{vanschoren2018meta}. However, there are key differences between the distribution shift phenomena in MLHS and that in traditional applications between the DS phenomena in MLHS and computer vision.
While distribution shift in computer vision lacks a theoretical model to describe~\cite{koh2021wilds}, that in MLHS comes from dynamical systems for which we have abundant knowledge on the resolved parts of the models. 
For example, let us consider using the projection method~\cite{weinan1995projection} to solve the incompressible Navier-Stokes equation.
Suppose we replace the Poisson pressure solver with a data-driven model based on neural networks.
Then, the dynamics of this data-driven hybrid numerical solver will
largely depend on the properties of the resolved numerical part of the
projection method, which is well-studied~\cite{weinan1995projection}.
By analyzing the stability properties of the
dynamics, one can quantify which family of distribution shifts may arise and
resolve them according to the information of the dynamics. Therefore, one can
pose the following question: \textit{Can we use the information of resolved parts to design robust learning algorithms for unresolved parts to improve hybrid simulation?}

In this paper, we first develop a mathematical framework to understand the
origin of distribution shift in MLHS and how it may lead to simulation performance deterioration. We emphasize
the difference between this instability issue in data-driven scientific computing
and distribution shift in the machine learning literature, such as reinforcement
learning and computer vision. Then, we propose an algorithm to improve the
simulation accuracy by mitigating distribution shifts.
We assume by manifold hypothesis~\cite{narayanan2010sample} that the correct
trajectories
lie on a low-dimensional manifold of the high-dimensional space, e.g. the fluid
configuration of the Navier-Stokes equation lies on the solution manifold of the
high-dimensional grid space. The key idea is to combine the physical information
of the resolved part of the model, i.e. Navier-Stokes equation and this manifold
structure learned from the data to form a regularization term for the
data-driven model. Intuitively speaking, this regularizer stabilizes the dynamics by preventing it from moving further away from the data manifold.
Such movements may result in configurations that are either
non-physical, or corresponds to different initial/boundary conditions.
Therefore, preventing such movements reduces the severity of
distribution shift, and the
data-driven model will stay relatively accurate during the simulation, which promotes high simulation fidelity in a long time interval.
In implementation, we first use an autoencoder (AE) to parameterize the underlying
data manifold. After this preprocessing, the AE is combined with the resolved dynamics and plugged into the loss
function of the model as a regularization, which prevents the simulation from
moving to unseen regimes. One then back-propagates this modified loss function
to optimize the data-driven model.
The algorithm is tested on several
representative numerical examples to demonstrate its effectiveness.
We show that the proposed approach can
improve simulation accuracy under different extents of distribution shift.
Indeed, the improvements become more significant in scenarios where the
fidelity of the simulation is highly sensitive to errors introduced in the data-driven
surrogates, such as transport-dominant fluid simulations and
reaction-diffusion equations with relatively strong non-linear reactions.
Specifically, in fluid simulations, naive
data-driven surrogate models may quickly cause error blow-ups,
but our method can maintain simulation accuracy over large time intervals.

The paper is structured as follows. In \cref{framework}, we establish a precise
framework to identify the origins of distribution shift in MLHS. In \cref{theory}, we introduce our regularized learning algorithm motivated
by the analysis of the previously identified form of
distribution shift and theoretically understand its performance in the linear setting.
We also discuss connections with the literature on control and system
identification. In \cref{numerics}, we validate our algorithm on several
practical numerical cases, including simulating a reaction-diffusion equation and the incompressible Navier-Stokes equation.

\section{Mathematical formulation of MLHS and the problem of distribution shift}\label{framework}

In this section, we provide the mathematical formulation of MLHS and then
identify the problem of distribution shift. In the first subsection, we give a
general treatment, and we provide concrete examples in \cref{sec:examples}.

\subsection{The resolved and unresolved components in hybrid simulations}

Throughout this paper, we consider the dynamics
\bequ\label{non-linear}
    \begin{aligned}
         \p_t\mfu = \mcL(\mfu, \mfy, t), \quad \mfy = \phi(\mfu, t),
    \end{aligned}
\eequ
where $\mfu$ is the resolved state variable, e.g. fluid velocity field
or chemical concentration field.
The vector $\mfy$ is the unresolved variable that drives the dynamics for
$\mfu$ but is either expensive to compute
or cannot be directly observed.
Throughout the paper, we will omit writing the explicit time-dependence
by using $\mfu$ ($\mfy$) to denote $\mfu(t)$ ($\mfy(t)$)
whenever there is no ambiguity.
In numerical simulation, one always discretizes the ODE or PDE to obtain a
finite-dimensional system, i.e. $\mfu \in \mbR^m, \mfy \in \mbR^n, \mcL: \mbR^m \times \mbR^n \times \mbR_+ \rightarrow \mbR^m, \phi: \mbR^m \times \mbR_+ \rightarrow \mbR^n$.
Therefore, we restrict the discussion to the finite dimension case where both
$\mcL, \phi$ becomes mapping between finite-dimensional spaces,
but most of the analysis here applies to the more general case of $\mfu, \mfy$
belonging to an infinite dimensional Hilbert space.
We adopt the following assumptions associated with~\cref{non-linear}:
\begin{itemize}
    \item 1. Resolved component: $\mcL$ is known, possibly non-linear.
    \item 2. Unresolved component: $\phi$ is either unknown or expensive to evaluate.
\end{itemize}
Given such a system, the goal is to first obtain information on the unresolved
model and then integrate it with the resolved part to simulate the whole dynamics.
This hybrid structure, which we call hybrid simulation problems, is general enough to include many applications, e.g. when simulating the trajectories of molecular dynamics
which satisfy some stochastic differential equations or the time-evolution of fluid velocity fields
following the Navier-Stokes equation. We focus on MLHS,
a particular variation of such hybrid simulation problems
where the unresolved component is tackled by a data-driven method.
This point will be made more precise in \cref{learning}.
To simplify the analysis, we adopt the forward Euler time discretization for simulation
\bequ
\begin{aligned}
\wht \mfu_{k+1} & = \wht \mfu_{k} + L(\wht\mfu_k, \wht\mfy_k)\Delta t,
\qquad
\wht \mfy_{k} = \phi(\wht \mfu_{k}),
\end{aligned}
\eequ
while other consistent discretizations may be analyzed in the same spirit. Here we drop the dependence of $L, \phi$ on $t$ to consider autonomous systems while our later 
discussion and algorithm are also applicable to non-autonomous systems.

\subsection{Examples of hybrid simulation problems}\label{sec:examples}
Before moving on to distribution shift, we first provide some examples in MLHS of our interest, which are also closely related to our numerical experiments. We will present 
these examples according to the structure \cref{non-linear}.

The first example is solving the 2D incompressible Navier-Stokes equation using the projection method~\cite{guermond2006overview,weinan1995projection},
a variant of which is as follows:
\begin{equation}\label{NS}
    \begin{aligned}
        	\frac{\p \mfu}{\p t} + (\mfu \cdot \nabla)\mfu -  \nu \Delta \mfu =   \nabla p,\quad \nabla \cdot \mfu = 0,\quad T \in [0, 1],
    \end{aligned}
\end{equation}
where $\mfu = (u(x, y, t), v(x, y, t))^T \in \mbR^2$ is the velocity and $p$
is the pressure. Fix a regular grid size for discretization. The projection
method can be written as follows
\bequ\label{projection}
    \begin{aligned}
        \mfu_{k+1} & = \mfu_k + L(\mfu_k, p_k )\Delta t = \mfu_k +
         (\nu \Delta \mfu_k
        - (\mfu_k \cdot \nabla)\mfu_k - \nabla p_{k})\Delta t,    \\
        p_{k} & = \phi(\mfu_k) = \Delta^{-1}(\nabla \cdot \lp \nu \Delta \mfu_k
        - (\mfu_k \cdot \nabla)\mfu_k\rp),   \\
    \end{aligned}
\eequ
We write $\Delta^{-1}$ as the formal inverse operator for the Laplace operator.
The resolved part is performed by first stepping the
convection and diffusion term then using pressure to correct the step, and the
unresolved part is a Poisson equation that relates the unsolved state $p_k$ with the
velocity $\mfu_k$.
Here, $p_{k}$ represents the unresolved variable $\mfy_k$ in our
formulation~\cref{non-linear}.

In this method, the most expensive step is the pressure computation,
which requires repeated solutions of similar large-scale linear equations.
State-of-the-art solvers such as multigrid~\cite{briggs2000multigrid}
and conjugate gradient~\cite{10.5555/865018} become prohibitively expensive
when the problem size is large.
Recently, a promising direction is replacing the Poisson solver by
data-driven surrogate models. Ref.~\cite{tompson2017accelerating}
replaces the pressure calculation step with a convolutional neural network
trained with unsupervised learning by requiring the
updated velocity to have zero divergence.
In the same spirit, machine learning models such
as tensor-based neural networks are used in~\cite{ling2016reynolds} to replace
classical turbulence modeling for unclosed Reynolds stress tensors. This model
is then plugged into the Reynolds-averaged Navier-Stokes simulation to predict
the flow separation. More recently,~\cite{pichi2023graph} combines the idea of
reduced-order modeling with graph convolutional neural network to encode the
reduced manifold and enable fast evaluations of parameterized PDEs.

The second example concerns simulations on two grids of different sizes.
Let us continue with the
projection method setting~\cref{projection}.
Denote the time-evolution operator on
grid size $n$ as $f_n: \mfu_k^n \rightarrow \mfu_{k+1}^n$, which becomes very
expensive to compute when $n$ is large.
Therefore, one may wish to replace
the fine-grid solver with a coarse-grid one and add some corrections.
This again reduces the computational burden for high-fidelity simulations.
This structure shows up in various scientific computing situations, such as multigrid
solvers~\cite{briggs2000multigrid}, Reynolds-averaged Navier-Stokes
(RANS)~\cite{pope2000turbulent, alfonsi2009reynolds}, and large eddy simulation
(LES)~\cite{zhiyin2015large}.
Suppose we have a fine grid with size $2n\times 2n$ and a coarse grid with size
$n \times n$, we use superscript
$n, 2n$ to denote the field defined on these two
grids, i.e. $\mfu_k^n, \mfu_k^{2n}$.
Moreover, we fix an interpolation and a restriction operator between fields on these two grids:
\bequ
	\begin{aligned}
	I_n^{2n}: \mbR^{n \times n} \rightarrow \mbR^{2n \times 2n}, \quad R_{2n}^{n}: \mbR^{2n \times 2n} \rightarrow \mbR^{n \times n}.
	\end{aligned}
\eequ
Now, we can state the second hybrid simulation problem as follows:
\bequ\label{coarse-fine}
	\lbb\begin{aligned}
		\mfu_{k+1}^{2n} & = I_n^{2n} \circ f_n(R_{2n}^{n}(\mfu_k^{2n})) + \mfy_k^{2n},		\\
		\mfy_k^{2n} & = \phi(\mfu_k^{2n}).
	\end{aligned}\right.
\eequ
In detail, the resolved component contains a solver of the evolution equation
over the coarse grid which marches forward by one time step, with the field given by
a restriction of the field on the fine grid.  Then, the next step field
configuration is interpolated to the fine grid and the unresolved component
serves to correct the deviation of the field variables between
simulating on grids of different sizes.
The unresolved variables $\mfy_k^{2n}$ can be calculated if we
have accurate simulation results $\mfu_k^{2n}$ on the fine grid by $\mfy_k^{2n} = \mfu_{k+1}^{2n} - I_n^{2n} \circ f_n(R_{2n}^{n}(\mfu_k^{2n}))$.
In the literature,~\cite{huang2022learning} treats the smoothing algorithm in
multigrid solver as unresolved components and other procedures such as
restriction and prolongation as resolved components. This fits into our framework
of MLHS.
They introduce a supervised loss function based on multigrid convergence theory
to learn the optimal smoother and improve the convergence rate over anisotropic
rotated Laplacian problems and variable coefficient diffusion problems. In
~\cite{maulik2019subgrid}, the authors apply neural networks with non-linear
regression to fit key fluid features and sub-grid stresses. Going beyond this,
~\cite{pedersen2023reliable} adds neural emulation to offline learning to
prevent the trajectories 
from deviating away from the ground truth. This data-driven model is plugged
into hybrid simulations to test its ability to preserve coherent structure and
scaling laws.

\subsection{Learning the unresolved model and distribution shift}
\label{learning}
We assume that the resolved part $\mathcal{L}$ is known,
or we at least have access to a gray box that can perform
its evaluation and compute its gradients.
While the unresolved part is unknown, we can
access data tuples $\lbb (\mfu_1, \mfy_1, t_1), (\mfu_2, \mfy_2, t_2),
\cdots, (\mfu_N, \mfy_N, t_N )\rbb$ which are obtained either by
physical experiments or accurate but expensive numerical simulations.
The goal is hence to construct an approximate mapping
$(\mfu_k, t_k) \mapsto \mfy_k$ from this dataset.
Viewing this as a supervised
learning task, there are many existing methods based on empirical risk
minimization~\cite{hastie2009overview, hastie2009elements, lecun2015deep}. One
of the simplest methods is to minimize the mean squared error. Given
a parameterized model $\phi_\theta$, e.g. a Gaussian
process or a neural network, we learn from data the
optimal value of the parameter $\theta$ that minimizes the error,
i.e. $\wht \theta = \arg\min_{\theta}
\mbE_{(\mfu,\mfy)} \norml \mfy - \phi_{\theta}(\mfu, t) \normr_2^2$. 
However, due to the dynamical structure, the i.i.d. assumption in statistical
learning~\cite{hastie2009elements} no longer holds in MLHS.
In particular, data belonging to the same dynamical trajectory will have a large correlation.
Moreover, an
appropriate performance evaluation is not necessarily the precise
identification of the system~\cite{ljung1998system} but the 
accuracy of the prediction of the trajectory over long times.
This is also known as a priori and
a posteriori performance in the closure modeling literature
~\cite{sanderse2024scientificmachinelearningclosure}.
Therefore, it is reasonable to expect that
simple empirical risk minimization
(which many existing MLHS works employ~\cite{stachenfeld2022learned, ren2022physics})
may not be the optimal method.
We will demonstrate this by using the least squares estimator
as a baseline comparison to our approach in~\cref{numerics}.

After determining the parameters of the unresolved model,
one can perform hybrid simulations
to obtain new trajectories, i.e.
\bequ\label{non-linear-simu}
    \begin{aligned}
    \wht \mfu_{k+1} = \wht \mfu_{k} + L(\wht\mfu_k, \wht\mfy_k, t_k)\Delta t,\quad \wht\mfy_k = \phi_{\wht \theta}(\wht\mfu_k, t_k).
    \end{aligned}
\eequ
Notice we add hat superscript to all the quantities related to this simulated dynamics to distinguish them from the ground truth $(\mfu_k, \mfy_k)$. We measure the performance 
of the simulator by the error along the whole trajectory, i.e.
\bequ\label{error-def}
    \min \sum_{k = 1}^n \norml \mfu_k - \wht \mfu_k \normr_2^2,
\eequ
where $\mfu_k$s belongs to the true trajectory and $\wht\mfu_k$s the simulated one with the same initial condition.

We begin by illustrating the issue of distribution shift by analyzing the trajectory error. For simplicity of presentation, we assume that the model hypothesis space
$\{\phi_\theta: \theta \in \Theta\}$ is bias-free,
meaning that there exists $\theta^* \in \Theta$
such that $\phi = \phi_{\theta^*}$. In the general case where the hypothesis space is universal but not closed,
this equality would be replaced by an approximate one,
but the argument follows analogously.
Comparing \cref{non-linear} and \cref{non-linear-simu}, we have
\bequ\label{decomp}
    \begin{aligned}
    \wht \mfu_{k+1} - \mfu_{k+1}        
    = & \ \wht \mfu_{k} - \mfu_{k} + \lp L(\mfu_k, \phi_{\wht\theta}(\mfu_k)) - L(\mfu_k, \phi_{\theta_*}(\mfu_k)) \rp \Delta t     \\
    & + \lp L(\wht\mfu_k, \phi_{\wht \theta}(\wht\mfu_k)) - L(\mfu_k, \phi_{\wht\theta}(\mfu_k)) \rp\Delta t.
    \end{aligned}
\eequ
The error introduced by moving one step forward is thus decomposed into two
parts: the error associated with the estimator $\wht \theta: L(\mfu_k,
\phi_{\wht\theta}(\mfu_k)) - L(\mfu_k, \phi_{\theta_*}(\mfu_k))$, and
the other error $L(\wht\mfu_k, \phi_{\wht \theta}(\wht\mfu_k)) - L(\mfu_k,
\phi_{\wht\theta}(\mfu_k))$. The former resembles the error appearing in
classical statistical inference problem under the i.i.d. setting since $\mfu_k$
follows an identical distribution of the training dataset. While the latter is different in the sense that $\wht \mfu_k$ and $\mfu_k$ may come from distributions of their 
respective driving dynamics, which are different since $\theta_*\neq \wht \theta$ due to noisy labels, few data points, or having an under-determined system. 
The difference between MLHS and statistical learning is also emphasized in~\cite{wang2021physics}.
The latter error $L(\wht\mfu_k, \phi_{\wht \theta}(\wht\mfu_k)) - L(\mfu_k,
\phi_{\wht\theta}(\mfu_k))$ is rather akin to the issue of stability
in numerical partial differential equations:
although the error in each step is relatively small,
it can accumulate exponentially as time step iterates,
driving $\wht \mfu_k$
to completely different regimes where we have not observed in the data $\mfu_k,
k=1,\cdots, N$. As a result,
the data-driven model $\wht\phi$ can no longer be trusted to be accurate,
since it is trained on data from a different distribution. This may then lead
to a vicious cycle, where further deterioration of trajectory error
occurs, leading to an increasing discrepancy between the distributions.
Consequently, one can understand the distribution
shift as a large discrepancy between the true distribution $\rho$ of $\mfu_k$
and simulated distribution $\wht \rho$ of $\wht \mfu_k$.
Concretely, we adopt the following definition for a simulation algorithm
\cref{non-linear-simu} to have distribution shift:
\begin{equation}\label{def:ds}
    \begin{aligned}
        \limsup_{t\rightarrow \infty}\frac{\log d(\rho_t, \wht \rho_t)}{t} \geq c > 0,
    \end{aligned}
\end{equation}
where $\rho_t, \wht \rho_t$ denote the distributions of the ground truth
$\mfu(t)$ and the simulation $\wht \mfu(t)$ starting from the initial condition
with same distribution $\rho_0$. $d(\cdot, \cdot)$ is an appropriate distance
between probability distributions. We emphasize that the definition of the
distribution shift relates to the time because of the dynamic nature of the
problem, which is different from traditional machine learning settings~\cite{koh2021wilds}.
The problem of distribution shift occurs commonly in MLHS.
For example,~\cite{ling2016reynolds, wu2019reynolds,poroseva2016accuracy} all
 observe unphysical solutions when substituting 
Reynolds stresses models learned from direct numerical simulation database into
RANS. Ref.~\cite{wu2019reynolds} attributes this to the ill-conditioning of the
RANS problems and propose a method to treat Reynolds stresses implicitly.
References~\cite{shankar2023differentiable2,shankar2023differentiable1}
deal with the subgrid-scale modeling in LES simulations by using neural
networks. Improved performance in simulations is achieved by taking into
consideration the whole trajectory accuracy in the loss function, which
implicitly combats distribution shift.
Despite several empirical successes, a general theoretical framework to understand
and combat distribution shift in MLHS is lacking.
The current paper aims to make progress in this direction.

Let us illustrate the problem of distribution shift using numerical examples.
We solve the 2D FitzHugh-Nagumo reaction-diffusion equation
over a periodic domain $[0,6.4]\times [0,6.4]$,
whose dynamic is given by
\begin{equation}\label{equ:RD}
    \begin{aligned}
        	\frac{\p \mfu}{\p t} & = D \Delta \mfu + \phi(\mfu), \quad T \in [0, 20], 	\\
		\phi(\mfu) & = \phi(u, v) = \begin{pmatrix}
			u - u^3 - v + \alpha	\\
			\beta(u - v)
		\end{pmatrix},
    \end{aligned}
\end{equation}
where $\mfu = (u(x, y, t), v(x, y, t))^T \in \mbR^2$ are two interactive
components, $D$ is the diffusion matrix, and $\phi(\mfu)$ is source
term for the reaction.
We fix the parameters to $\alpha=0.01,\beta=1.0,D=\begin{pmatrix}
    0.05 & 0    \\
    0 & 0.1
\end{pmatrix}$, where the system is known to form Turing patterns~\cite{murray2003mathematical}.
Assuming that we do not know the
exact form of the reaction term $\phi(\mfu)$,
the resolved part is the diffusion term and the unresolved part is the
non-linear reaction term.
The ground truth data is
calculated using the semi-implicit Crank-Nicolson scheme with fully explicit
discretization for the non-linear term.
Assuming we do not know that the reaction term is pointwise in space,
we use a convolutional neural network-based model to learn it
based on ordinary least squares (OLS).
Specifically, we solve the
least squares problem $\min_{\theta}\mbE_{\mfu}\norml \phi_{\theta}(\mfu) -
\phi(\mfu)\normr_2^2$ using Adam~\cite{kingma2014adam}. Moreover, we conduct an
ablation study where the unresolved model consists of the exact form of the
non-linear term as in~\cref{equ:RD}, plus a random Gaussian noise whose variance
is set equal to the optimization error of the OLS estimator, which is $10^{-4}$
in the following experiments.
By comparing this and the simulation
based on the OLS model,
we can identify the source of the accumulating trajectory error as either the estimation error of 
the unresolved model or the shift in the distribution of the input data fed to the unresolved model.
More details on the data generation and optimization procedure is described in subsection 2.1 and 2.2 of supplementary materials .
We simulate this PDE with a test initial condition
sampled from the same distribution as training trajectories and compare the
results with the ground truth. We present the
comparison in~\cref{RD-ds}.
\begin{figure}[ht]\label{RD-ds}
          \centering
          \centerline{\includegraphics[width=.6\linewidth]{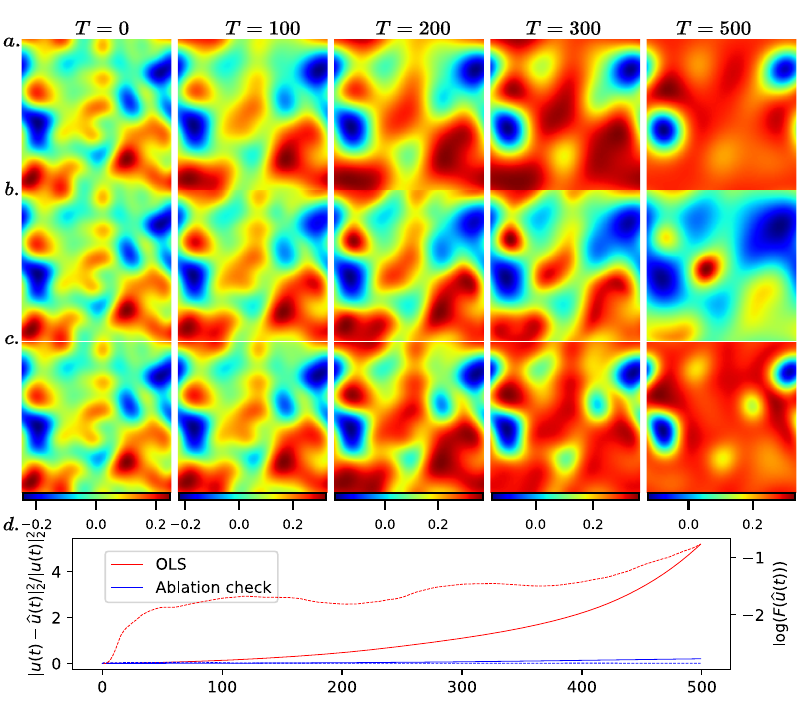}}
          \caption{The first three rows display the snapshots of
ground truth, simulated fields, and ablation study simulation
at time steps: 0, 100, 200, 300, and 500. (d) plots the relative error (solid line) and distribution shift (dash line) between two flow configurations at the corresponding time. 
The error is defined as $\frac{\norml \wht\mfu_t - \mfu_t \normr_2}{\norml \mfu_t \normr_2}$ at each time step $t$ instead of the whole trajectory as in \cref{error-def}.
The plotted distribution shift
roughly measures the average discrepancy between the distributions of $\mfu_k$ and
$\wht \mfu_k$ at time step $k$,
and is calculated using an autoencoder explained in \cref{alg-intro}. Although the OLS estimator and ablation-study estimator share the same error magnitude for unresolved models, their long-time 
trajectory performance forms sharp contrasts, indicating the optimization error of the unresolved model could not explain the failure in long-time trajectory prediction.}
\end{figure}

The first three rows display the snapshots of
ground truth, simulated fields, and ablation study simulation
at time steps: 0, 100, 200, 300, and 500. (d) plots the relative error (solid curve) and distribution shift (dotted curve) between two flow configurations at the corresponding time. The error 
is defined as $\frac{\norml \wht\mfu_t - \mfu_t \normr_2}{\norml \mfu_t \normr_2}$ at each time step $t$ instead of the whole trajectory as in \cref{error-def}.
The plotted distribution shift
roughly measures the average discrepancy between the distributions of $\mfu_k$ and
$\wht \mfu_k$ at time step $k$,
and is calculated using an autoencoder explained in \cref{alg-intro}.
As can be observed, it does not take a very long
time for the distribution shift issue to be severe enough using the OLS estimator. Comparing the OLS error and distribution shift with those of the
ablation study, we conclude that
the estimation error is not the main cause of the trajectory error,
since the ablation-study trajectory has the same magnitude of estimation
error, but has a much smaller trajectory error and distribution shift.
Moreover, in both cases the relative error of the trajectory has the same trend as the distribution shift, indicating correlations.

Indeed, this phenomenon is ubiquitous in modern
MLHS~\cite{wang2021physics, zhang2019active, stachenfeld2022learned}.
As a second example, we solve the Navier-Stokes equation \cref{NS} using
projection method \cref{projection} where the pressure solver is replaced by a learned convolutional neural network
predictor. The predictor is optimized over data pairs of fluid velocity and pressure obtained from high-fidelity simulation to attain an error of $10^{-5}$. Due to the complexity of the projection solver and 
grid issue, we do not conduct an ablation study in this case.
\begin{figure}[ht]\label{NS-ds}
    \centering
    \centerline{\includegraphics[width=.6\linewidth]{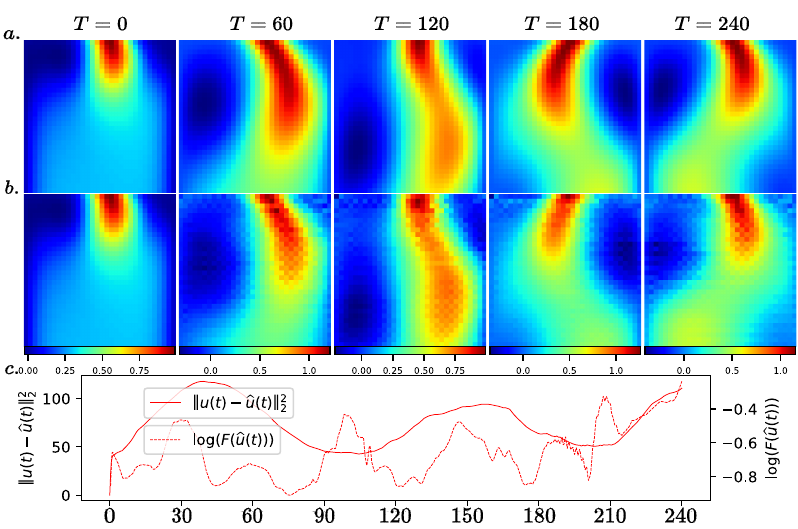}}
    \caption{
    Configurations of the Navier-Stokes equation simulation with Reynolds number $200$ at time step: 0, 60, 120, 180, 240: (a) ground truth (b) simulated fluid field using OLS estimator. 
    We only show the velocity field over part of the domain. The dynamic blows up in the last configuration, which does not show up in the case of reaction-diffusion equations. This
    can be explained by the fact that the Navier-Stokes equation is more unstable than reaction-diffusion equations. In (c), we illustrate the error
    along the trajectories, the distribution shift.
    }
\end{figure}
The distribution shift in \cref{NS-ds} turns out
to be more severe when we simulate the incompressible Navier-Stokes equation. After $160$ time steps, the fluid configuration
becomes rough and un-physical. Here, we again observe that the DS and error increase
together, suggesting correlations.
In the next section, we will make this connection precise,
and develop principled methods to control
the distribution shift, thereby improving the prediction fidelity.

\section{Theoretical analysis and algorithm}
\label{theory}
In this section, we first put the distribution shift issue into a theoretical framework. Then, we introduce and analyze our algorithm in this section. We will provide rigorous
study of the linear case and use this to motivate the algorithm.
\subsection{Distribution shift in linear dynamics and motivation of \algName}
As there exist few tools to tackle the general non-linear dynamics
\cref{non-linear}, we switch to the simpler case of
linear dynamics, which also capture some key features in general situations. We consider the
following the hybrid simulation problem
\bequ\label{linear}
    \begin{aligned}
        \frac{d\mfu}{dt} & = A\mfu + B\mfy,  \quad \mfu\in\mbR^m, \mfy \in \mbR^n, A\in \mbR^{m \times m}, B\in \mbR^{m \times n}   \\
        \mfy & = C^* \mfu,  \quad C^* \in \mbR^{n\times m}.
    \end{aligned}
\eequ
Such structures appear in many scenarios, e.g. finite difference solution
of discretized linear PDEs and linear control problems.
Alternatively, one can think
of this linear system as the linearization of the
non-linear dynamics~\cref{non-linear} about some steady-state.

We assume that the state variables $\mfu_i$s in training data come from several
trajectories and $\mfy_i$s may be subject to some random
measurement error, i.e. $\mfy_i = C^*\mfu_i + \epsilon_i$,
or in matrix form $Y = C^*\mfU + \epsilon$.
The key point is that these trajectories may belong to a low-dimensional
subspace in the high-dimensional state space.
This is the case for many scientific computing problems, e.g. parametric
elliptic PDE~\cite{bachmayr2017kolmogorov} and is
also related to the notions of Kolmogorov N-width and the manifold hypothesis
~\cite{fefferman2016testing}. Although transport-dominated problems with
low-decaying rates of the 
Kolmogorov N-width are difficult to approximate using linear reduced-order
models, a large amount of 
literature~\cite{Holmes_Lumley_Berkooz_Rowley_2012, holmes1997low} show
evidence of the existence of coherent structures in turbulence and their
low-dimensional nature. Moreover, 
recent work~\cite{eivazi2020deep} using nonlinear model reduction has achieved
efficient reduction over this type of problems.

We use $V \subset \mbR^m$ to denote the subspace that
contains all the training data $\mfu_i$. In this setting, the appearance of this
low-dimensional structure may be caused by two different situations. In the
first situation, the initial value of $\mfu$ is supported on the subspace
spanned by several eigenvectors of the evolution operator $e^{(A+BC^*)}$, i.e.
\begin{align}
    \mfu_0 \in \operatorname{span}\{\mfv_1, \mfv_2, \cdots, \mfv_l \}, \quad e^{(A+BC^*)}\mfv_1 = \eig_1 \mfv_1, \cdots, e^{(A+BC^*)}\mfv_l = \eig_l \mfv_l,
\end{align}
Then, all the training data $\mfu_i$ belong to this subspace, i.e. $V
= \operatorname{span}\{\mfv_1, \mfv_2, \cdots, \mfv_l \}$. The second situation
is that the dynamics is degenerate in the sense that the
evolution matrix $A$ is not of full rank.
This will not
appear in the differential formulation as infinitesimal transformations
in the form $e^{(A+BC^*)\Delta t}$ are always
non-degenerate. However, for a discrete dynamic $\mfu_{k+1} = A\mfu_k + B\mfy_k, \quad \mfy_k = C^* \mfu_k$,
if the matrix $A+BC^*$ is degenerate, then all the data $\mfu_k$ will belong to
the range of this operator $A+BC^*$, which is a proper subspace.
Most cases of scientific computing applications belong
to the first class.
Due to the degeneracy in the training data,
the least squares solution is not unique.
Hence, various empirical regularizations (e.g. $\ell^2$ regularization)
can be introduced to obtain an estimator with desirable properties.

Now, suppose one obtains an estimator $\wht C$, one can calculate
several error metrics.
The first is the test error
\begin{equation}\label{stats-error}
    l_{\OLS}(\wht C) = \mbE\norml (\wht C - C^*) \mfu \normr_2^2,
\end{equation}
where $\mfu$ is sampled from the low dimensional subspace $V$ and follows the
same distribution as the training data.
We refer to this as the statistical estimation error or a priori error.
However, we are interested in the a posteriori error
of the simulated dynamics
\bequ\label{simu-linear}
    \begin{aligned}
        \frac{d\wht\mfu}{dt}  = A\wht\mfu + B\wht\mfy,  \quad  \wht\mfy = \wht C \wht\mfu,
    \end{aligned}
\eequ
compared to the ground truth $\mfu$ with the same initial condition.
The time evolution of their difference is given by

\begin{equation}\label{decomp-linear}
    \begin{aligned}
        \frac{d}{dt} (\wht\mfu - \mfu) = & \ (A+B\wht C)\wht \mfu - (A+BC^*)\mfu      \\
        = & \  (A+B\wht C)\opP_V (\wht \mfu - \mfu) + B(\wht C - C^*)\mfu + (A+B\wht C)(\wht\mfu - \opP_V\wht \mfu),
    \end{aligned}
\end{equation}
where $\opP_V$ is the orthogonal projection onto the data subspace $V$, and we have
$\opP_V\mfu = \mfu$ since $\mfu$ belongs to the data subspace. As derived in \cref{decomp-linear}, the overall error of the trajectory during evolution
can be decomposed into three parts. In the following, we make a detailed
analysis of each term.

The first term can be viewed as an amplitude and stability factor of the error propagation of the simulated dynamics inside the data manifold $V$.
If we assume the ground truth dynamics is stable in $V$, while not necessary to be stable over the whole state space, then this term can be simply
bounded by the spectral information of the generator $(A+B\wht C)\opP_V$, see details in \cref{pf:a-posteriori-error}.

The second term is simply the statistical estimation error of the estimator
$\wht C$ defined in~\cref{stats-error}. Given that the initial
condition $\mfu_0$ of the training and test trajectories are sampled from the
same distribution, this term is well-bounded in classical statistical learning
theory~\cite{hastie2009elements}, provided the number of data tuples is
sufficiently large.
In particular, this error has nothing to do with the distribution shift.

Now we move to the last part $(A+B\wht C)(\wht \mfu - \opP_V \wht \mfu)$,
which can be more easily understood from a geometric viewpoint.
The last factor $(\wht \mfu - \opP_V \wht
\mfu)$ measures how far the simulated trajectory is away from the data subspace
$V$, i.e. $\norml \wht \mfu - \opP_V \wht \mfu \normr_2 =
\dist(\wht \mfu, V)$. In other words,
this term measures how far the simulated trajectory's distribution ``shifts''
from the true distribution.
This is exactly the term we want to control to maintain simulation fidelity.
This term forms a sharp
contrast with the two terms previously mentioned in the sense that it is not automatically
bounded in most algorithms for estimating $\wht C$.
We will illustrate this point
further by deriving an error bound of the simulated dynamics, showing its sensitive dependency on the
distribution shift term. Thus, we can improve a posteriori accuracy if we combat such distribution shifts.

Now, we can formally write the error propagation along the simulated dynamics as
\begin{equation}\label{equ:tr-error-prop}
    \begin{aligned}
        & \ \wht\mfu(T) - \mfu(T) \\
        = & \ \int_0^T e^{(A+B\wht C)\opP_V(T-t)} \lp B(\wht C - C^*)\mfu(t) +(A+B\wht C)(\wht \mfu(t) - \opP_V\wht\mfu(t))\rp dt.
    \end{aligned}
\end{equation}
The exponential factor $e^{(A+B\wht C)\opP_V(T-t)}$ corresponds to the first
term while two terms in the bracket that follows correspond to the second and
the third terms mentioned above. The distribution shift term in MLHS, i.e. the
second term in the bracket is far different from classical machine learning concepts
such as covariate shift and label shift. In computer vision, the cause of
distribution shift may be extrinsic~\cite{koh2021wilds}, i.e. the pictures in
training sets are all taken during the daytime while those for testing are all
taken at night.
Such shifts are hard to model, so their resolution
tends to depend on data augmentation and related techniques,
e.g. the Dagger algorithm~\cite{10.1145/3054912, pmlr-v15-ross11a},
where instead of changing the optimization
problem, the data source is modified. This shares many similarities with
adversarial training, which also adds more data to the training set to make the
prediction robust under adversarial attack, a typical distribution shift in the
area of computer vision~\cite{goodfellow2020generative}.
In contrast, in our setting, the distribution shift is intrinsically
driven by the hybrid simulation structure of which we have partial knowledge.
Hence, we can quantify this distribution shift and design specific
algorithms to mitigate them.

Returning to the linear problem,
in order to guarantee that the error $\wht \mfu - \mfu$ is bounded
over the simulated trajectory, a natural choice is to use a regularization for
$\norml \wht \mfu - \opP_V \wht \mfu \normr_2$. The most naive choice would be
$\norml \wht \mfu - \opP_V \wht \mfu \normr_2$ itself, but a problem is that $\wht \mfu$ is calculated via hybrid simulation until time $t$, which corresponds to a rather complicated computational graph.
This makes subsequent gradient-based optimization computationally expensive. Hence, we
take advantage of a one-step predictor, where
we set $\wht \mfu_{k+1}= \mfu_k+ (A\mfu_k + B\wht C\mfu_k)\Delta t$,
meaning that $\wht \mfu_{k+1}$ is calculated based on a single step of the simulated
dynamics from the ground truth solution $\mfu_k$.
Assuming $\mfu_k \in V$, one has
\begin{equation}
    \norml \wht \mfu_{k+1} - \opP_V \wht \mfu_{k+1} \normr_2 = \Delta t\norml \opP_{V^{\perp}}(A\mfu_k + B\wht C\mfu_k)\normr_2,
\end{equation}
and we may penalize the right-hand side to promote trajectory alignment with the data subspace.
Wrapping up all the ingredients, we state the loss function in the linear case as
\begin{equation}\label{reg-obj}
    \min_{C} l(C) := \min_C \mbE_{(\mfu, \mfy)} \lp \norml \mfy - C \mfu\normr_2^2 + \lambda \norml \opP_{V^{\perp}}(A + BC)\mfu \normr_2^2 \rp,
\end{equation}
and here $\lambda$ is a penalty strength parameter that can be
chosen during training.
We briefly discuss in subsection 2.5 of supplementary materials on choosing $\lambda$.

In the above discussion, we have assumed that the underlying data subspace $V$ is known a
priori, or we can calculate $\opP_{V^{\perp}}$ directly. However, in most cases, the
only information we have is a set of resolved state variables $\mfu_i$s. Then,
we can use standard dimension reduction methods such as principal component
analysis~\cite{hastie2009elements} to obtain a subspace $V$, which can be thought
of as an approximation of the data subspace. The loss objective \cref{reg-obj} can
thus be calculated.

\subsection{Algorithm in the general form}\label{alg-intro}

The discussion above applies to linear dynamics where the data also lies in a linear subspace.
In this subsection, we generalize it to non-linear cases and
state the general form of our algorithm.

Let us return to our original formulation of the task~\cref{non-linear},
where we have a dataset $\lbb (\mfu_1, \mfy_1, t_1),
(\mfu_2, \mfy_2, t_2), \cdots, (\mfu_N, \mfy_N, t_N )\rbb$.
As before, we assume that the resolved variable $\mfu$ lies on a $l$-dimensional manifold $\mcM \subset \mbR^m$ in
the configuration space and the estimator $\phi_{\theta}(\mfu, t)$ achieves high accuracy along this manifold while it has no guarantee outside the
manifold. The key difference between the general and the linear dynamics is that $\mcM$ may not be linear and we do not have a linear projection onto it.

To resolve this, we leverage an autoencoder to represent the
data manifold structure: we learn two neural networks
$E: \mbR^m \rightarrow \mbR^l$
(encoder) and $D: \mbR^l \rightarrow \mbR^m$ (decoder)
to represent the manifold. $E, D$ are
simultaneously optimized to minimize the reconstruction error
\begin{equation}\label{ED-opt}
    \min_{E, D}\frac{1}{N}\sum_{i=1}^N F^2(\mfu_i), \quad F(\mfu_i) := \norml \mfu_i - D(E(\mfu_i))\normr_2.
\end{equation}
Omitting optimization issues, we can assume that after the training, the function $F$ vanishes on $\mcM$ and a greater value of $F$ implies a greater deviation to the manifold. Thus $F$ can be used as an indicator of
distribution shift and is used in all the figures in this paper to quantify
it.

\begin{comment}
\begin{equation}
    \begin{aligned}
        F(\mfu) = & \ \sqrt{\sum_{i=l+1}^m (\mfu^T \mfv_i)^2},  \\
        \nabla F(\mfu) = & \ \frac{\sqrt{\sum_{i=l+1}^m (\mfu^T \mfv_i)^2}}{F(\mfu)},\\
        L(\mfu_i, \phi_{\theta}(\mfu_i), t_i) = &\ A\mfu_i + B\phi_{\theta}(\mfu_i).
    \end{aligned}
\end{equation}
\begin{equation}
    F(\mfu^{t+1}) = F(\mfu^t) + \nabla F(\mfu^t)^T L(\mfu^t, \phi(\mfu^t)) = \nabla F(\mfu^t)^T L(\mfu^t, \phi(\mfu^t)).
\end{equation}
\end{comment}
The nonlinear analog of the regularization term in \cref{reg-obj} should be the projection to the normal space $N_{\mfu_i}\mcM = T_{\mfu_i}^{\perp}\mcM$ at $\mfu_i$. However, constructing this projection operator would require a highly accurate optimized decoder map \cite{kvalheimshould} so that $dD_{\mfz_i}(\mfe_i), i=1,..., l$ forms a basis of the tangent space $T_{\mfu_i}\mcM$. Here $dD_{\mfu_i}$ is the differential of the mapping $D$ at $\mfz_i$ and $\{\mfe_i, i=1,...l\}$ is the standard basis of $\mbR^l$. To reduce the computational cost, we pick the normal vector $\nabla F(\mfu_i) \in N_{\mfu_i}\mcM$ and use $(\nabla F(\mfu_i))^T L(\mfu_i, \phi_{\theta}(\mfu_i), t_i)$ to measure the deviation of the velocity along this normal direction. The fact that $\nabla F(\mfu_i)$ belongs to the normal space of $\mcM$ at $\mfu_i$ directly follows from the assumption that $F$ vanishes on $\mcM$ after training. Moreover, we define $F(\mfu_i) := \norml \mfu_i - D(E(\mfu_i))\normr_2$ instead of $\norml \mfu_i - D(E(\mfu_i))\normr_2^2$ since the former has an $O(1)$ gradient even near $\mcM$ while the latter's gradient is of order $o(1)$ near $\mcM$ which may cause the regularization to be ineffective. The overall regularized loss function of our algorithm in general cases reads
\begin{equation}\label{ours}
    \min_{\theta} l(\theta) := \min_{\phi_{\theta}} \frac{1}{N}\sum_{i=1}^N \lp \norml \mfy_i - \phi_{\theta}(\mfu_i) \normr_2^2 + \lambda \lp \lp \nabla F(\mfu_i) \rp^T L(\mfu_i, \phi_{\theta}(\mfu_i), t_i)\rp^2 \rp.
\end{equation}
In the second term, $D, E$ are optimized in advance and frozen during the training of $\phi_{\theta}$. Since the loss function is regularized by a 
term that approximately measures the deviation of the velocity vector $L(\mfu_i, \phi_{\theta}(\mfu_i), t_i)$ 
from the tangent space of the data manifold, we name our algorithm ``MLHS with \algName''. The overall algorithm is summarized in \cref{alg-pseudocode}.

There is a subtle difference in our algorithm between linear and nonlinear
problems. While in linear case we can directly construct the projection
operators onto the orthogonal complement of the linear manifold
denoted by $\opP_{V^{\perp}}$, in nonlinear case each function 
$F(\mfu)$ only contributes one vector normal to $\mcM$. A technique to 
obtain more normal vectors in the nonlinear case is to have several $F_i$ composed by different autoencoders trained with different architectures or initializations. In this case, the regularization term will 
have the form
\begin{equation}
    \lambda \norml \lp \nabla \mfF(\mfu_i) \rp^T L(\mfu_i, \phi_{\theta}(\mfu_i), t_i)\normr_2^2 = \lambda \sum_{j=1}^K\lp \lp \nabla F_j(\mfu_i) \rp^T L(\mfu_i, \phi_{\theta}(\mfu_i), t_i)\rp^2.
\end{equation}
Here $K$ is number of normal directions chosen to be penalized with a maximum value $K_{\max} = m - l$. Notice that a greater $K$ also implies a greater computational and memory cost.
Throughout our experiment, we will only use one autoencoder to construct the regularization term as this already provides a 
significant improvement in the simulation accuracy.

\begin{algorithm}\label{alg-pseudocode}
\caption{MLHS with \algName}
\begin{algorithmic}[1]
\INPUT{$\lbb (\mfu_1, \mfy_1, t_1), \cdots, (\mfu_N, \mfy_N, t_N )\rbb$, resolved model, penalty strength $ \lambda$.}
\STATE{Learn a parameterized model which encodes the structure of training data, i.e. $F_{\eta}(\mfu) \geq 0, F_{\eta}(\mfu_k) = 0.$ }
\STATE{Freeze the parameters of this learned model.}
\STATE{Introduce another surrogate model $\phi_{\theta}(\mfu)$.}
\FOR{$k=1,2,\cdots, N$}
    \STATE{Predict the control variable $\wht \mfy_k = \phi_{\theta}(\mfu_k)$.}
    \STATE{Calculate the state variable after one-step iteration, i.e. $\wht \mfu_{k+1} = \mfu_{k} + \Delta t L(\mfu_k, \wht \mfy_k, t_k)$.}
    \STATE{Form the loss $l(\theta) = \mbE\lb \norml \mfy_k - \wht \mfy_k \normr_2^2 + \lambda \lp\lp \nabla F(\mfu_k) \rp^T L(\mfu_k, \wht \mfy_k, t_k)\rp^2 \rb$.}
    \STATE{Backpropogate to update $\theta$.}
\ENDFOR
\OUTPUT{\algName: $\phi_{\theta}$.}
\end{algorithmic}
\end{algorithm}

Compared to the literature on distributionally
robust optimization (DRO)
~\cite{rahimian2019distributionally} and numerical analysis of dynamical systems
~\cite{stuart1994numerical, stuart1998dynamical}, our work takes advantages from both.
First, the \algName \ belongs to the family of DRO in
the sense that it minimizes the data-driven module of the unresolved part over
some perturbations of the data distribution.
However, unlike the original DRO which
considers all possible perturbations of the data distribution
in its neighborhood under some metric, e.g. Wasserstein metric~\cite{blanchet2022confidence},
our regularization focus
on those perturbations caused by the resolved components of the dynamics.
These facilitate us to take advantage of the partial
knowledge of the dynamics.
As a result, the regularizer we proposed is more problem-specific. Moreover, from the perspective of DRO, our method takes into consideration 
the behavior of estimator under certain
perturbation normal to the
data manifold, which is much more tractable than guarding against
arbitrary perturbations in distribution space. In \cref{numerics}, we will implement other benchmarks with 
general regularizations and compare with our algorithm to illustrate the 
benefits of combining resolved model information into regularization.

Furthermore, we can understand our method through the lens of the stability analysis of
numerical methods for differential equations.
The classical result that the least squares estimator converges
to the ground truth in the limit of large datasets
can be understood as a consistency statement.
Here, we show that this is insufficient, and to
ensure convergence of the simulated trajectories
one needs to also promote stability,
and the regularizer proposed by us serves precisely this role.
This regularization approach should be contrasted with recent works~\cite{yu2021onsagernet,cranmer2020lagrangian,lin2020data,greydanus2019hamiltonian} on
learning dynamical models from data which builds stability
by specifying the model architecture.

\subsection{Interpretation and analysis of the algorithm}

Let us now quantify in the linear case, the gains of our algorithm over
the ordinary least squares estimator.
We first calculate the exact formula for two estimators.
Due to the low-rank structure, the OLS algorithm allows
infinitely many solutions.
Throughout this section, we make the common choice
of the minimum 2-norm solution. Let us first recall two important choices considered in this section will be the OLS estimator and \algName, i.e.
\begin{equation}\label{app-loss-func}
    \begin{aligned}
    C_{\OLS}  &:= \arg\min_C \frac{1}{N}\sum_{i=1}^N \norml \mfy_i - C \mfu_i\normr_2^2,  \\
    C_{\abbrvalgName}   &:= \arg\min_C \frac{1}{N}\sum_{i=1}^N \lp \norml \mfy_i - C \mfu_i\normr_2^2 + \lambda \norml \opP_{V^{\perp}}(A + BC)\mfu_i \normr_2^2 \rp.
    \end{aligned}
\end{equation}
The next proposition provides the explicit formula for them.
\begin{proposition}\label{estimator}
    Consider the linear dynamics \cref{linear}, the training data of the variable
    $\mfu$ is arranged into a data matrix $\mfU \in \mbR^{m \times N}$.
    Moreover, we assume the column space of $\mfU$ is contained in a subspace
    $V\subset \mbR^m$ with associated projection operator $\opP_V$.
    Then, the OLS estimator
    and \algName \ for the unresolved component $Y =
    C^* \mfU +\epsilon$ have the following form:
    \begin{equation}\label{estimator-formula}
    \begin{aligned}
    \wht C_{\OLS} = & \ C^* \opP_V + \epsilon \mfU^{\dagger},    \\
    \wht C_{\abbrvalgName} = & \ (\mfI + \lambda B^T \opP_{V^{\perp}} B)^{-1}(C^* \opP_V + \epsilon \mfU^{\dagger} - \lambda B^T \opP_{V^{\perp}} AP_{V}).
    \end{aligned}
\end{equation}
\end{proposition}
Specifically, in the noiseless scenarios, $\wht C_{\OLS}$ recovers $C^*$ on the subspace $V$ but vanishes on its orthogonal complement. To interpret our estimator, let us consider a 
simpler case with $B = \mfI, A = \mathbf{0}$. This corresponds to the dynamic $\p_t \mfu = C^*\mfu$,
and the regularized estimator is given by
\begin{equation}
    \wht C_{\abbrvalgName} = (\mfI + \lambda \opP_{V^{\perp}})^{-1}(C^* \opP_V + \epsilon \mfU^{\dagger})= (\mfI + \lambda \opP_{V^{\perp}})^{-1}\wht C_{\OLS}.
\end{equation}
Therefore, our estimator performs a weighted least squares regression which penalizes
the direction perpendicular to the data subspace $V$, controlled by $\lambda$.
As the error along this normal direction is the main cause of the distribution shift,
it is sensible to reduce it more than the error along the tangent direction,
which causes deviations inside the manifold.
This shows that the \algName \ does mitigate distribution shift.

\begin{proof}[Proof of \cref{estimator}]
We solve the linear regression problem of the unresolved component analytically. Recall that the solution for 
the $\min_C \norml Y - CU \normr_{\rm F}^2$ is given by $\wht C = YU^{\dagger}$ 
where $U^{\dagger}$ is the pseudo-inverse of the matrix $U$. Consequently,
assume $Y$ is noiseless, i.e. $Y = C^* U$, we have $\wht C = C^* U U^{\dagger}$.
Let us assume the number of data samples is greater than the dimension of the
state, i.e. $U  \in \mbR^{n \times N}, n < N$. Then, if $U$ is full rank, i.e.
$\rank U = n$, the product $U U^{\dagger}$ is exactly  the identity matrix and
$\wht C$ thereby recovers the $C^*$ exactly. However, in our setting, the data
is supported on a low dimensional subspace and the data matrix is
rank-deficient. Therefore, $U U^{\dagger} = \opP_V$,
where $V$ is the column space of the data matrix $U$ and also the support set of it. We have $\wht C = C^* \opP_V$
in this case. Similarly, in the noisy case $Y = C^*U + \epsilon$, we have
\begin{equation}
    \wht C_{\OLS} = C^* \opP_V + \epsilon U^{\dagger}.
\end{equation}
This is the OLS estimator of the unresolved component.

Next, we calculate the formula for \algName. The loss function can be reduced to
\begin{equation}
    \begin{aligned}
    & \ \norml Y - C\mfU \normr_{\rm F}^2 + \lambda \norml \opP_{V^{\perp}}(A + BC)\mfU \normr_{\rm F}^2            \\
    = & \ \Tr[(Y-C\mfU)^T(Y-C\mfU)] + \lambda \Tr[(\opP_{V^{\perp}}(A + BC)\mfU)^T(\opP_{V^{\perp}}(A + BC)\mfU)]   \\
    = & \ \Tr[(Y-C\mfU)^T(Y-C\mfU)] + \lambda \Tr[\mfU^T(A + BC)^T\opP_{V^{\perp}}(A + BC)\mfU],
    \end{aligned}
\end{equation}
where we use the fact $\opP_{V^{\perp}}^T \opP_{V^{\perp}} = \opP_{V^{\perp}}$ since $\opP_{V^{\perp}}$ is a projection. Consequently, the first-order condition is given by
\begin{equation}
    C\mfU\mfU^T -Y\mfU^T + \lambda B^T \opP_{V^{\perp}}(A+BC) \mfU\mfU^T) = 0.
\end{equation}
We finally derive our estimator as
\begin{equation}
    \begin{aligned}
    \wht C_{\abbrvalgName} = & \ (\mfI + \lambda B^T \opP_{V^{\perp}} B)(Y\mfU^T - \lambda B^T \opP_{V^{\perp}} A\mfU\mfU^T)(\mfU\mfU^T)^{\dagger}     \\
    = & \ (\mfI + \lambda B^T \opP_{V^{\perp}} B)^{-1}(C^* \opP_V + \epsilon U^{\dagger} - \lambda B^T \opP_{V^{\perp}} AP_{V}).
    \end{aligned}
\end{equation}
\end{proof}

\begin{remark}
    Unlike classical $L^1$ or $L^2$-regularization, the tangent-space
    regularization does not have the issue of bias-variance trade-off,
    as both the least square and the regularization terms admit a common
    minimizer $C^*$. This is easily observed by plugging in the ground truth 
    $C^*$ to the formula of the regularizer, giving
    \begin{equation}
        \norml \opP_{V^{\perp}}(A + BC^*)\mfu_i \normr_2^2 = 0,
    \end{equation}
    since $\mfu_i$ lies in the invariant subspace $V$ of the generator of the dynamics $A+BC^*$. Therefore, increasing $\lambda$ will mitigate the effect of distribution shift, while at the same time 
    it does not deteriorate the performance of  the least squares problem.
\end{remark}

Next, we move on to analyze the accuracy gain of the whole
simulation algorithm by proving a bound for the error
of the simulation trajectory.
Denote the statistical error of \algName \ by
\begin{equation}\label{equ:tr-error}
    l_{TR}(\wht C) = \mbE\lp \norml (C^* - \wht C) \mfu\normr_2^2 + \lambda \norml \opP_{V^{\perp}}(A + B\wht C)\mfu \normr_2^2 \rp.
\end{equation}
We define the following function which helps simplify the notation
\begin{equation}
    Q_m(r, T) = m^2\int_0^T (2 + t^{m-1})e^{r t}dt \leq 3m^2 (1\vee T^{m})(1\vee e^{rT}), \quad T > 0,
\end{equation}
where $a \vee b := \max(a, b).$
In all the discussions that follow, we use $\eig_{\max}(A)$ to denote the greatest real part of the eigenvalues of $A$, i.e.
\begin{equation}
    \eig_{\max}(A) = \max\{\Re(s): \exists v \neq \mathbf{0}, Av = sv\}.
\end{equation}
As there are several assumptions for the theorem, we state it here for clarity.
\begin{assumption}\label{ass:theorem}
    We assume both the OLS estimator and $\abbrvalgName$ estimator $\wht C_{\OLS}$, $\wht C_{TR}$ are optimized to
    have error bounded by $\delta$ so that $l_{\OLS}(\wht C_{\OLS}), l_{TR}(\wht
    C_{TR}) < \delta$ in \cref{stats-error} and \cref{equ:tr-error}.
    We assume also that
    $\norml \opP_{V^{\perp}}(A+B\wht C_{\abbrvalgName})\opP_V \normr_2$ $ < \sqrt{\delta/\lambda}$.
\end{assumption}
Without the second assumption, the error bounds in \cref{thm:a-posteriori-error} hold with high probability via the observation
\begin{equation}
    \opP_{V^{\perp}}(A+B\wht C_{\abbrvalgName})\opP_V \mfu = \opP_{V^{\perp}}B(\mfI + \lambda B^T \opP_{V^{\perp}} B)^{-1}\epsilon \mfU^{\dagger},
\end{equation}
and classical theory of linear regression~\cite{hastie2009elements} if the training data has variance $\Sigma \geq c\mfI, c>0$. In \cref{ass:theorem}, we include this 
assumption to simplify the proof.

Moreover, we define several notations to simplify 
the derivation. Let $P_1, P_2, P_3$ be matrices which
transform $A+B\wht C$, $(A+B\wht C)\opP_V$, and $\opP_{V^{\perp}}(A+B\wht C)$ into their Jordan normal forms respectively and
\begin{equation}
    \begin{aligned}
        c_1 & = \cond(P_1), &&c_2 = \cond(P_2), && c_3 = \cond(P_3),       \\
        e_1 & = \eig_{\max}(A+B\wht C), && e_2 = \eig_{\max}((A+B\wht C)\opP_V), && e_3 = \eig_{\max}(\opP_{V^{\perp}}(A+B\wht C)),
    \end{aligned}
\end{equation}
where $\cond(P)$ denotes the condition number of matrix $P$.
\begin{theorem}\label{thm:a-posteriori-error}
    Under the same setting as \cref{estimator} and \cref{ass:theorem}, suppose the
    true trajectory is simulated from an initial condition $\mfu(0)\in V$ which
    follows the same distribution of the training data. Then, the errors of OLS
    and our algorithm is bounded respectively by
    \begin{equation}\label{equ:a-posterior-error}
        \begin{aligned}
            & \ \mbE\norml \wht\mfu_{\OLS}(T) - \mfu(T) \normr_2 \leq c_1\sqrt{\delta}\norml B \normr_2Q_m(e_1, T),      \\
            & \ \mbE\norml \wht\mfu_{\abbrvalgName}(T) - \mfu(T) \normr_2 \leq c_2\sqrt{\delta}\Big(\norml B \normr_2 Q_m(e_2, T) \\
        & \qquad  + \frac{9m^4c_3}{\sqrt{\lambda}}\lp 1 + 3m^2c_1\sqrt{\delta}\norml B \normr_2\rp\norml A+B\wht C \normr_2 (1\vee T^{3m})(1\vee e^{e_1T}) \Big).
        \end{aligned}
    \end{equation}
\end{theorem}

Generally speaking, the exponential factor $e_1, e_2$ in $Q_m$ may make both bounds loose, especially for large $T$. This is
inevitable as our bounds deal with worst-case scenarios and exponential growth is common in unstable dynamics. 
If otherwise the dynamic is stable, $Q_m$ will decay exponentially and the bound becomes more effective.

Comparing the two bounds, the first difference is the exponents appearing in the $Q$-function, i.e. in OLS it is $e_1=\eig_{\max}(A+B\wht C)$
while in our algorithm it is $e_2=\eig_{\max}((A+B\wht C)\opP_V)$. This is consistent with our expectation as the data-driven surrogate can only be trusted in the subspace $V$. Therefore,
if the ground truth dynamics is unstable in the orthogonal complement $V^{\perp}$ while stable in $V$, meaning that $e_1 > 0 > e_2$ the error of OLS 
will be dominated by the error made in $V^{\perp}$.
For tangent-space regularization, the better exponent $e_2$ clearly mitigates this. This is also observed quantitatively in \cref{fig:linear-cmp}: comparing the error propagations 
of the $\abbrvalgName$ with $\lambda=10^7$ with OLS, the rates are very different at the beginning with $\abbrvalgName$ close to $e_2$ while OLS close to $e_1$. However, we need to
be careful with the additional terms appearing in the second bound that may have a big $e_1$ exponent which may lead
to the blow-up of the error. However, it also contains an extra factor $\sqrt{\frac{1}{\lambda}}$ and will be much smaller than the counterpart appearing in OLS bound as long as the 
penalty $\lambda$ is large enough.
This is justified from \cref{fig:linear-cmp},
where we observe $\abbrvalgName$ trajectories with increasing $\lambda$s spend more time with a slower rate $e_2$, and then saturate to the $e_1$-rate.

\begin{figure}[ht]\label{fig:linear-cmp}
    \centering
    \centerline{\includegraphics[width=.6\linewidth]{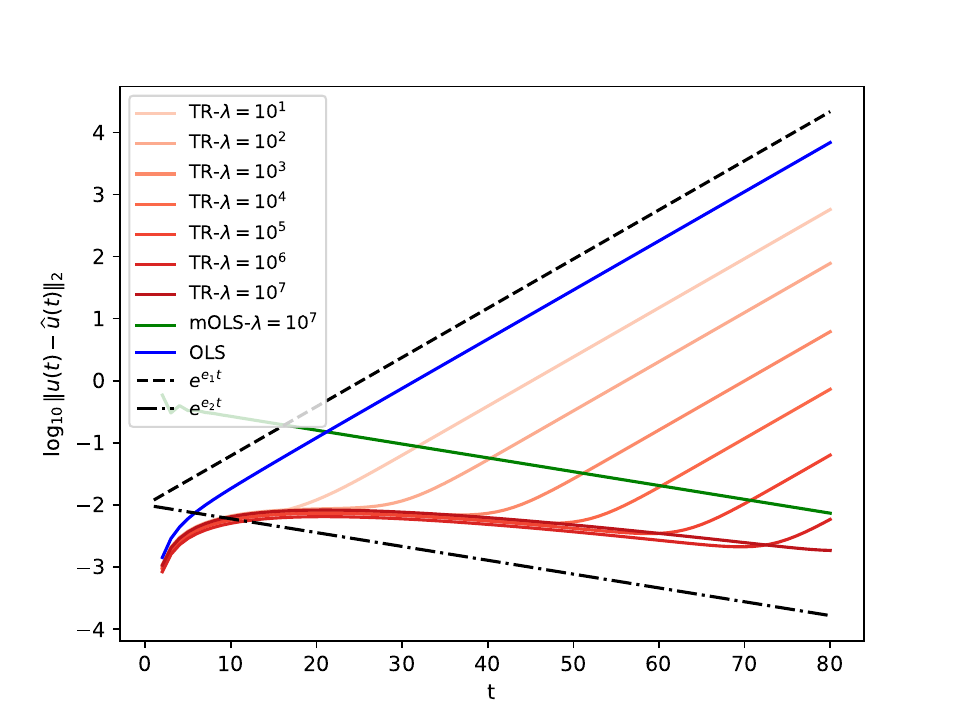}}
    \caption{We compare the simulation error of various estimators in a synthetic linear dynamic where
    the data subspace $V(V^{\perp})$ is the stable (unstable) manifold. This quantitatively supports our theory from three perspectives: 
    1). The trajectory error of the OLS simulator coincides with the rate $e_1$ over the unstable manifold; 2). The trajectory error of our algorithm has two stages, 
    i.e. $e_2$-rate at the beginning and $e_1$ later, which correspond exactly to two parts in \cref{equ:a-posterior-error}. 3). As $\lambda\rightarrow\infty$, the length of the 
    first stage increases, meaning a smaller simulation error.}
\end{figure}

\begin{proof}[Proof of \cref{thm:a-posteriori-error}]\label{pf:a-posteriori-error}
In the proof, we will use $\norml \cdot \normr$ to denote 2-norm for simplification. We start with the error propagation equation which does not take into consideration the low-dimensional 
structure in the data and dynamics.
\begin{equation}\label{equ:error-prop-1}
    \begin{aligned}
        \frac{d}{dt} \lp \wht\mfu - \mfu \rp
        = & \ \lp A+B\wht C\rp\lp\wht \mfu - \mfu\rp + B\lp\wht C - C^*\rp\mfu,
    \end{aligned}
\end{equation}
which can be solved exactly to obtain 
\begin{equation}
    \wht\mfu(T) - \mfu(T) = \int_0^T e^{(A+B\wht C)(T-t)} B(\wht C - C^*)\mfu(t)dt.
\end{equation}
Therefore, using the bound (proof in supplementary materials) that for $M \in \mbR^{m\times m}$, $\norml e^{Mt} \normr \leq \cond(P)m^2(2 + t^{m-1})e^{\eig_{\max}(M) t}$,
we obtain
\begin{equation}\label{equ:ols-error}
    \begin{aligned}
        \mbE\norml \wht\mfu_{\OLS}(T) - \mfu(T) \normr \leq & \ \int_0^T \norml e^{(A+B\wht C)(T-t)} \normr \norml B \normr \mbE\norml (\wht C - C^*)\mfu(t) \normr dt      \\
        \leq & \ c_1\sqrt{\delta}\norml B \normr Q_m(e_1, T) \lesssim \sqrt{\delta}Q_m(e_1, T).
    \end{aligned}
\end{equation}
Notice that this bound also holds for $\mfu_{\abbrvalgName}$ as it only requires 
$\mbE\norml (C_{\abbrvalgName} - C^*)\mfu(t) \normr < \delta$ which is guaranteed by our assumption for $l_{\abbrvalgName}$.
Next, for tangent-space regularization, we need to take into account the existence of the low-dimensional subspace $V$ and the projection operator $\opP_V$. Recall the decomposition of the error propagated 
dynamics
\begin{equation}
    \begin{aligned}
        \frac{d}{dt} (\wht\mfu - \mfu) = (A+B\wht C)\opP_V (\wht \mfu - \mfu) + B(\wht C - C^*)\mfu + (A+B\wht C)(\wht\mfu - \opP_V\wht \mfu),
    \end{aligned}
\end{equation}
and the integration formula for the propagation error \cref{equ:tr-error-prop}. We need to control the term $\opP_{V^{\perp}}\wht\mfu(t) = \wht \mfu(t) - \opP_V\wht\mfu(t)$ originating 
from the subspace decomposition.
\begin{equation}
    \begin{aligned}
        \frac{d}{dt}\opP_{V^{\perp}}\wht\mfu  
        = & \ \opP_{V^{\perp}}(A+B\wht C)\opP_{V^{\perp}} \wht\mfu  + \opP_{V^{\perp}}(A+B\wht C)\opP_V\lp \wht\mfu - \mfu \rp + \opP_{V^{\perp}}(A+B\wht C)\mfu,
    \end{aligned}
\end{equation}
whose solution can be organized as 
\begin{equation}
    \begin{aligned}
        & \ \opP_{V^{\perp}}\wht\mfu(T) \\
        = & \ \int_0^T e^{\opP_{V^{\perp}}(A+B\wht C)(T-t)} \lp \opP_{V^{\perp}}(A+B\wht C)\opP_V\lp \wht\mfu(t) - \mfu(t) \rp + \opP_{V^{\perp}}(A+B\wht C)\mfu(t) \rp dt.
    \end{aligned}
\end{equation}
Now, using the \cref{ass:theorem}, one has
\begin{equation}
    \begin{aligned}
        & \ \mbE\norml \opP_{V^{\perp}}\lp \wht\mfu_{\abbrvalgName}(T)\rp \normr \\
        \leq & \ c_3\sqrt{\frac{\delta}{\lambda}} \lp Q_m(e_3, T) + m^2\int_0^Te^{(T-t)e_3}(2+(T-t)^{m-1}) \mbE\norml \wht\mfu_{\abbrvalgName}(t) - \mfu(t) \normr  \rp    \\
        \leq & \ c_3\sqrt{\frac{\delta}{\lambda}}\lp 3m^2 (1\vee T^{m})(1\vee e^{e_3T}) + 9m^4c_1\sqrt{\delta}\norml B \normr (1\vee T^{2m})(1\vee e^{e_1T}) \rp       \\
        \leq & \ 3m^2c_3\sqrt{\frac{\delta}{\lambda}}\lp 1 + 3m^2c_1\sqrt{\delta}\norml B \normr\rp(1\vee T^{2m})(1\vee e^{e_1T}).
    \end{aligned}
\end{equation}
In the third line, we use the bound \cref{equ:ols-error}.
Now we plug in the bounds of $\opP_{V^{\perp}} \wht\mfu_{\abbrvalgName}(T)$ to the error propagation equation and obtain the following bound
\begin{equation}\label{equ:tr-alg-error}
    \begin{aligned}
        & \ \mbE\norml \wht\mfu_{\abbrvalgName}(T) - \mfu(T) \normr \leq c_2\sqrt{\delta}\Big(\norml B \normr Q_m(e_2, T) \\
        & \qquad  + \frac{9m^4c_3}{\sqrt{\lambda}}\lp 1 + 3m^2c_1\sqrt{\delta}\norml B \normr\rp\norml A+B\wht C \normr (1\vee T^{3m})(1\vee e^{e_1T}) \Big).
    \end{aligned}
\end{equation}
\end{proof}

\begin{remark}
    Recall our definition of the distribution shift in \cref{def:ds}. In the proof, assuming $e_1 > 0 > e_2$,
    we obtain the following estimates of the simulation error of OLS and our algorithm:
    \begin{equation}
        \begin{aligned}
            \mbE\norml \opP_{V^{\perp}}\lp \wht\mfu_{\OLS}(T)\rp \normr & = O((1\vee T^{m})(1\vee e^{e_1T})), \\
            \mbE\norml \opP_{V^{\perp}}\lp \wht\mfu_{\abbrvalgName}(T)\rp \normr & = O(\frac{(1\vee T^{3m})(1\vee e^{e_1T})}{\sqrt{\lambda}}).
        \end{aligned}
    \end{equation}
    Therefore, if we define
    \begin{equation}
        d(\rho, \wht \rho) = \Big| \mbE\norml \opP_{V^{\perp}}\mfu \normr - \mbE\norml \opP_{V^{\perp}}\wht \mfu \normr \Big|, \quad \mfu \sim \rho, \wht \mfu \sim \wht \rho,
    \end{equation}
    we conclude that the OLS algorithm suffers from distribution shift as $\opP_{V^{\perp}}\mfu = \mathbf{0}$ while our algorithm can mitigate this issue by choosing a large penalty 
    $\lambda$, which is also observed quantitatively in \cref{fig:linear-cmp}.
    Although this constant may be exponentially large over $T$ in unstable linear dynamics, our numerical experiments in the next section imply that $\lambda$ needs not to be very 
    large practically.
\end{remark}

\section{Numerical experiments}\label{numerics}

In this section, we present numerical experiments which illustrate the
effect of the tangent-space regularized algorithm presented in~\cref{alg-pseudocode}.
We use the experiments to illustrate two phenomena.
The first demonstrates the effectiveness of our algorithm, which
we test over reaction-diffusion equations and the incompressible
Navier-Stokes equation, and compare with benchmark algorithms using the least
squares, as well as simple regularizations that do not account
for the structure of the resolved dynamics or the data manifold.
For the second goal, we show that our method brings more significant
improvements in problems with severe distribution shifts.
This is demonstrated on the same family of PDEs, i.e. reaction-diffusion equations with
varying diffusion coefficients and the Navier-Stokes equation with different Reynolds numbers.
These varying parameters can be understood - as we show numerically - as
quantitative indicators of the severity of distribution shift. The implementation of our method and experiment reproduction is found in
the repository~\cite{code-repo}.

\subsection{Tangent-space regularized algorithm improves simulation accuracy}

We first show that the proposed algorithm can be advantageous in several prototypical machine-learning augmented scientific computing applications.

In the first experiment, we use the FitzHugh-Nagumo reaction-diffusion equation \cref{equ:RD} as the test case of our algorithm. The resolved-unresolved structure is the 
coarse-fine grid correction structure discussed in \cref{sec:examples}. This helps us obtain high-fidelity simulation results by combining a low-fidelity numerical solver with a 
fast neural network surrogate model for correction. Here we choose 
$\alpha=0.01,\beta=1.0,D=\begin{pmatrix}
    \gamma & 0    \\
    0 & 2\gamma
\end{pmatrix}, \gamma \in \{0.05, 0.10, 0.15, 0.20, 0.25\}$ in the simulation which guarantees the existence of Turing patterns. The initial condition $\mfu_0$ is generated by i.i.d. sampling from a normal
distribution $u(x, y, 0), v(x, y, 0) \sim \mcN(0, 1)$.
This system is known to have dynamic spatial patterns where initial smaller
clusters gradually transform and merge to form bigger clusters according to the balance between diffusion and reaction effects~\cite{murray2003mathematical}. In literature, the importance of the 
sampling procedure is also emphasized, and various adaptive sampling methods based on error and particle methods are studied in~\cite{doi:10.1137/22M1527763, wen2023coupling, zhao2022adaptive}.

Now, we explain how the training data is prepared. Recall that in
\cref{coarse-fine} we discussed the hybrid simulation structure for
 a correction procedure under a pair of coarse and fine mesh:
\bequ
	\lbb\begin{aligned}
		\mfu_{k+1}^{2n} & = L(\mfu_k^{2n}, \mfy_k^{2n}) = I_n^{2n} \circ f_n(R_{2n}^{n}(\mfu_k^{2n})) + \mfy_k^{2n},		\\
		\mfy_k^{2n} & = \phi(\mfu_k^{2n}).
	\end{aligned}\right.
\eequ
The mapping $\phi$ between the fine grid state $\mfu_k^{2n}$ to the correction term $\mfy_k^{2n}$ is the unresolved model we desire to learn. Here $I_n^{2n}, R_{2n}^{n}$ are fixed restriction and interpolation operators
between two meshes, i.e.
\begin{equation}\label{equ:coarse-fine}
    \begin{aligned}
        & R_{2n}^{n}: \mbR^{2n\times 2n} \rightarrow \mbR^{n\times n}, \quad \mfu_{ij}^n = \frac{\mfu_{2i,2j}^{2n}+\mfu_{2i,2j+1}^{2n}+\mfu_{2i+1,2j}^{2n}+\mfu_{2i+1,2j+1}^{2n}}{4},     \\
        & I_{2n}^{n}: \mbR^{n\times n} \rightarrow \mbR^{2n\times 2n}, \quad \mfu_{2i,2j}^{2n} = \mfu_{2i,2j+1}^{2n} = \mfu_{2i+1,2j}^{2n} = \mfu_{2i+1,2j+1}^{2n} = \mfu_{ij}^n.
    \end{aligned}
\end{equation}
To obtain training data for $\mfu_n^t, \mfu_{2n}^t$, we
simulate \cref{equ:RD} (with the same initial condition)
on two meshes of sizes $64 \times 64$ and $128 \times 128$
respectively,
with time step $\Delta t = 0.01, N_{\text{step}} = T/\Delta t = 100$,
using a second-order Crank-Nicolson scheme. It remains to calculate the labels for the training data,
i.e. $\mfy_k^{2n} = \phi(\mfu_k^{2n})$.
These are obtained by using the first half of the \cref{equ:coarse-fine}.
Finally, we pair them
together to form the training data $\lbb (\mfu_n^1, \mfy_n^1), \cdots,
(\mfu_n^N, \mfy_n^N) \rbb$.
The data here is different from the one introduced in~\cref{learning} since the system is autonomous, so we do not include time indices.
To summarize, the training data
contains two patches of field component patterns
of size $100\times 128 \times 128$ corresponding to
$\mfu, \mfv$ respectively and two patches $\mfy_u, \mfy_v$ of the same shape and generated
according to \cref{equ:coarse-fine}.

The baseline algorithms we implemented to compare with ours is $\phi_{\OLS}$,
which minimizes the mean-squared error on the training data, i.e.
\bequ\label{ols}
	\phi_{\OLS} = \arg\min \mbE_{\mfu \sim \rho_0}\norml \phi_{\theta}(\mfu) - \phi(\mfu) \normr_2^2,
\eequ
where we use $\rho_0$ to denote the distribution of the field variable. The corresponding numerical results are labeled by ``OLS''.

\begin{figure}[ht]\label{RDcomp}
          \centering
          \centerline{\includegraphics[width=.6\linewidth]{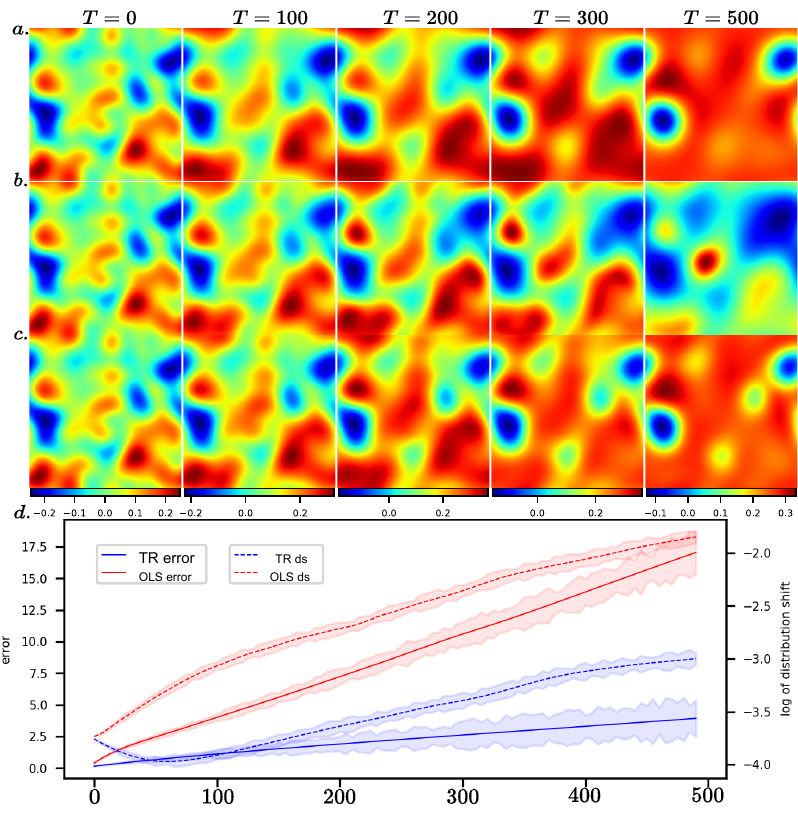}}
          \caption{
            Simulation of reaction-diffusion equations and configurations at time step: 0, 100, 200, 300, 500, (a) ground truth; (b) simulated fluid field using OLS estimator; (c) 
            simulated fluid field using $\abbrvalgName$ estimator. (d) is the comparison of the ordinary least squares and regularized one on reaction-diffusion equations: the solid line 
            represents the averaged error along the trajectory and the dashed line represents the averaged distribution shift calculated using function $F$ defined in \cref{ED-opt}. The shadow part is the standard deviation in 
            ten random experiments with different training and testing data. Both the error and the distribution shift increase at a lower speed in $\abbrvalgName$ simulation than in OLS simulations. This 
            suggests that the better performance of $\abbrvalgName$ algorithm in reaction-diffusion equations is partially caused by improving the distribution shift issue.}
\end{figure}
We summarize the comparison results in~\cref{RDcomp}.
Here, solid lines represent the errors along the trajectories
and dashed lines represent the distribution shift.
The shaded part is the standard deviation over ten random experiments
with different splits of the training and testing data.
As can be observed, compared to our method the OLS baseline incurs
greater errors and distribution shifts.
Moreover, this difference increases with time.
In contrast, consistent with our analysis in the simplified setting,
our algorithm mitigates both the distribution shift and
the trajectory error.

Next, we consider the Navier-Stokes equation \cref{NS} as the test case of our algorithm \cref{alg-pseudocode}. In this case, the resolved model is an explicit scheme to evolve the velocity field while 
unresolved model is calculated via a fast CNN surrogate, which can accelerate the simulation for large grid systems. The simulation domain is rectangular with an aspect ratio $1/4$. On the upper and
lower boundary we impose the no-slip boundary condition on the velocity and on the
inlet, i.e. $x=0$ we specify the velocity field $\mfu$. Moreover, at the outlet, i.e. $x=n$ we pose zero-gradient
conditions on both the horizontal velocity and pressure, i.e. $\frac{\p p}{\p n}
= \frac{\p u}{\p n} = 0$. The inflow boundary is defined according to following
functional forms:
\bequ
    \mfu(0, y, t) = \begin{pmatrix}
        u(0, y, t)\\
        v(0, y, y)
    \end{pmatrix} = \begin{pmatrix}
        \exp\{-50(y-y_0)^2\} \\
        \sin t \cdot \exp\{-50(y-y_0)^2\}
    \end{pmatrix}
\eequ
The training configuration is calculated using the projection method
~\cite{weinan1995projection} with a staggered grid, i.e. the pressure is placed at
the center of each cell and horizontal (vertical) velocity is placed at
vertical (horizontal) edge of each cell.
The jet location $y_0$ is selected from $0.3$ to $0.7$ uniformly.

Again, we compare the performance of our algorithm and the OLS method. As
illustrated in~\cref{NS-cmp}, naively using OLS to estimate the unresolved models
will generally lead to an error blow-up, whereas using our regularized
estimator we can obtain a reasonable simulation along the whole time horizon.
Moreover, the trend of the trajectory error fits that of the distribution
shift. In particular, the error and distribution shift blow-ups occur simultaneously.
Our algorithm again decreases the trajectory error by mitigating the distribution shift problem.
Moreover, the blow-up of the error in the simulated dynamics is not explained
in \cref{thm:a-posteriori-error}
as the theory is established for linear dynamics while finite-time blow-up is
known to be a special feature of nonlinear dynamics.

We also notice the unphysical noise appearing in the numerical experiments of
the Navier-Stokes equation, i.e. \cref{NS-cmp} b. which is
also observed in~\cite{ling2016reynolds, ling2015evaluation, wu2019reynolds,
poroseva2016accuracy} when performing data-driven RANS modeling.
While this is not the main focus of this paper which deals with the
distribution shift issue, we observe that such noise mainly occurs in the
simulation of the NS equation instead of reaction-diffusion equations. 
Therefore, one possible explanation is that this noise is caused by
transport-dominated equations and diffusion can reduce this noise in
diffusion-dominated equations.

\begin{figure}[ht]\label{NS-cmp}
          \centering
          \centerline{\includegraphics[width=.6\linewidth]{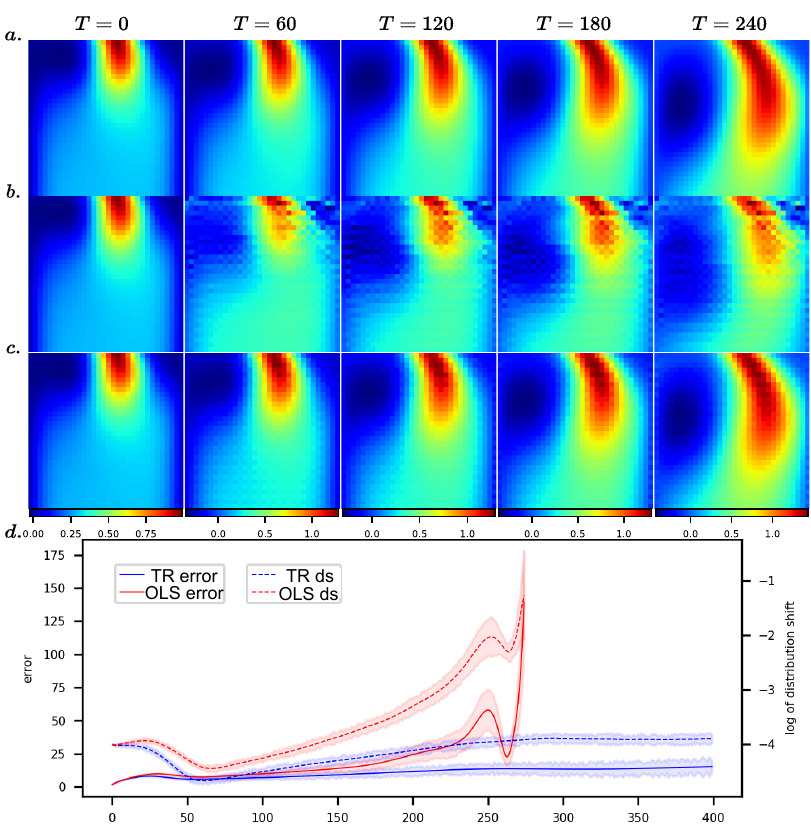}}
          \caption{
          Simulation of the Navier-Stokes equation and configurations at time step: 0, 60, 120, 180, 240, (a)
          ground truth; (b) simulated fluid velocity field using OLS estimator;
          (c) simulated fluid velocity field using $\abbrvalgName$ estimator;
          (d) the comparison of the OLS estimator and the regularized one.
          The solid lines represent the trajectory error and the dashed lines represent the distribution shift calculated usingusing function $F$ defined in \cref{ED-opt}.
          We observe that the OLS estimator leads to error (and DS)
          blow-up at around 270 steps, while our method remains stable.
          Again, the trends of error and distribution shift are highly correlated, and our method that controls distribution shift leads to error control,
          consistent with our analyses.}
\end{figure}

The machine-learning augmented NS simulations appear to have
a greater degree of distribution shift,
and our method has a relatively stronger benefit in this case.
This suggests that our method is most beneficial
in problems where the intrinsic distribution shift is severe.
In the next subsection, we will study this
more systematically by testing our method on
the reaction-diffusion equations with varying
diffusion coefficient $\gamma$ and
Navier-Stokes equations with varying Reynolds
numbers.

\subsection{Performance under distribution shift of different magnitudes}\label{numerical-cmp}
We now demonstrate concretely that our method brings
more significant improvements when the intrinsic distribution
shift is more severe.
Intuitively speaking, since our algorithm penalizes
the deviation of the simulated trajectories to the data manifold while OLS ignores this, it would be reasonable to expect that our
algorithm will outperform OLS more in problems where such deviations are large. A sanity check showing different distribution shift extents under different parameters is provided in Figure 3 of supplementary 
materials.

Here, we introduce two more baselines for ablation studies. Compared with the OLS
algorithm, these algorithms both add some regularization to the loss function
while their regularizations are general and not dynamics-specific compared to
our regularization. Consequently, any improved performance in our algorithm over
these regularized algorithms suggest that targeted regularization, which
takes into account the structure of the resolved part, effectively balances accuracy and stability.

The first baseline regularizes the original least squares objective by a term that quantifies some complexity of the network 
model, e.g.
\bequ\label{m-ols}
	\phi_{\text{mOLS}} = \arg\min \mbE_{\mfu \sim \rho_0}\norml \phi_{\theta}(\mfu) - \phi(\mfu) \normr_2^2 + \lambda\norml \theta \normr_2^2.
\eequ
Here the naive $L^2$ regularized algorithm is implemented by adjusting the weight decay~\cite{loshchilov2017decoupled} 
parameter of the Adam solver during the training process. The corresponding numerical results are labeled by ``mOLS'' 
(``m'' for modified).

The second baseline adds noise to the training data during training, which can be formulated as
\bequ\label{a-ols}
	\phi_{\text{aOLS}} = \arg\min \mbE_{\mfu \sim \rho_{\xi}}\norml \phi_{\theta}(\mfu) - \phi(\mfu) \normr_2^2.
\eequ
Notice that the training data follows the distribution $\rho_{\epsilon} = \rho_0
* \xi$ where $*$ denotes the convolution operator and $\xi$ is some random noise which we
take to be Gaussian random variables with an adjustable variance. In our implementation, we generate the noise and add it to the input during each epoch, so 
the inputs are not the same across epochs. Moreover, this noise $\xi$ should
be distinguished from the observational noise $\epsilon$ we introduced in
\cref{theory}. This method has been used in several previous works
~\cite{stachenfeld2022learned, DBLP:journals/corr/abs-2010-03409} to improve
stability and prevent the error from accumulating too fast. As this technique is
similar to adversarial attack and data augmentation~\cite{data-aug,
NIPS2014_5ca3e9b1, Goodfellow-et-al-2016}, we name this estimator and algorithm
as ``aOLS'' (``a'' stands for adversarial). We slightly abuse the terminology here by 
naming this random perturbation as some kind of adversarial attack.

% \begin{table}[ht]\label{RD-table}
% \normalsize
% \centering
% \caption{Performance on reaction-diffusion equations over 10 experiments.}
% \begin{tabular}{|c|ccccc|} \hline
%      & $\gamma = 0.05$ & $\gamma = 0.10$ & $\gamma = 0.15$ & $\gamma = 0.20$ & $\gamma = 0.25$  \\
%     \hline
%     & & & \textbf{grid size: }$64 \times 64$ & &   \\
%     \hline
%     OLS & 8.40E+1 & 1.04E+1 & 8.99E+0 & 2.22E+0 & 6.78E+1 \\
% mOLS & 8.90E+0 & 9.22E+0 & 1.41E+0 & 2.26E+0 & 9.98E+1 \\
% aOLS & 3.57E+1 & 4.00E+1 & 2.40E+0 & 1.79E+0 & 1.06E+0 \\
% $\abbrvalgName$ & \textbf{1.23E+1} & \textbf{1.06E+1} & \textbf{1.38E+1} & \textbf{1.85E+1} & \textbf{2.15E+1} \\
%     Diff & 99.9\% & 99.0\% & 98.5\% & 91.7\% & 68.3\%  \\    \hline
%     & & & \textbf{grid size: }$128 \times 128$ & & \\
%     \hline
%     OLS & 4.10E+2 & 8.18E+1 & 5.11E+1 & 4.88E+0 & 2.25E+0 \\
% mOLS & 2.40E+1 & 3.20E+1 & 2.20E+0 & 1.07E+1 & 1.25E+0 \\
% aOLS & 1.46E+2 & 2.58E+2 & 8.69E+0 & 3.30E+0 & 1.79E+0 \\
% TR & \textbf{2.79E+1} & \textbf{1.66E+1} & \textbf{2.42E+1} & \textbf{1.81E+1} & \textbf{1.91E+1} \\
% \hline
% Diff & 99.9\% & 99.8\% & 99.5\% & 96.3\% & 91.5\% \\
%     \hline
% \end{tabular}
% \end{table}

\begin{table}[tb]
  \Large
  \caption{Performance on reaction-diffusion equations over 10 experiments.}
  \label{RD-table}
  \centering
  \resizebox{0.95\columnwidth}{!}{
  \begin{tabular}{p{2cm} p{3cm} p{3cm} p{3cm} p{3cm} p{3cm}}
  \toprule [1.5pt]
  \parbox{2cm}{  } &   $\gamma = 0.05$ & $\gamma = 0.10$ & $\gamma = 0.15$ & $\gamma = 0.20$ & $\gamma = 0.25$   \\ \midrule[1.5pt]
  & & \textbf{grid size:} & $64 \times 64$ & &  \\\midrule[0.5pt]
  OLS & 8.40E+1 & 1.04E+1 & 8.99E+0 & 2.22E+0 & 6.78E-1 \\ \midrule[0.5pt]
  mOLS & 8.90E+0 & 9.22E+0 & 1.41E+0 & 2.26E+0 & 9.98E-1 \\ \midrule[0.5pt]
  aOLS & 3.57E+1 & 4.00E+1 & 2.40E+0 & 1.79E+0 & 1.06E+0 \\ \midrule[0.5pt]
  $\abbrvalgName$ & \textbf{1.23E-1} & \textbf{1.06E-1} & \textbf{1.38E-1} & \textbf{1.85E-1} & \textbf{2.15E-1} \\ \midrule[0.5pt]
  Diff & 99.9\% & 99.0\% & 98.5\% & 91.7\% & 68.3\%  \\  \midrule[0.5pt]
  & & \textbf{grid size:} & $128 \times 128$ & &  \\\midrule[0.5pt]
  OLS & 4.10E+2 & 8.18E+1 & 5.11E+1 & 4.88E+0 & 2.25E+0 \\ \midrule[0.5pt]
mOLS & 2.40E+1 & 3.20E+1 & 2.20E+0 & 1.07E+1 & 1.25E+0 \\ \midrule[0.5pt]
aOLS & 1.46E+2 & 2.58E+2 & 8.69E+0 & 3.30E+0 & 1.79E+0 \\ \midrule[0.5pt]
TR & \textbf{2.79E-1} & \textbf{1.66E-1} & \textbf{2.42E-1} & \textbf{1.81E-1} & \textbf{1.91E-1} \\ \midrule[0.5pt]
Diff & 99.9\% & 99.8\% & 99.5\% & 96.3\% & 91.5\% 
  \\\bottomrule[1.5pt]
  \end{tabular}
  }
\end{table}

In~\cref{RD-table},
we test the algorithms under different diffusion coefficient $\gamma = 0.05$, 
$0.10, 0.15, 0.20, 0.25$.  We calculate the relative trajectory error at time $T = 1000$, i.e.  $\frac{\norml
\mfu_k - \wht \mfu_k \normr_2}{\norml
\mfu_k \normr_2}$ over ten random splits of the training and test data. The last row calculates
the performance difference between our method and the OLS baseline.
We observe that under almost all the settings our method out-performs
OLS, mOLS, aOLS baselines, and comparing horizontally,
we see that the relative improvement
decreases with $\gamma$, as the diffusion term mitigates the distribution shift.

We now consider the Navier-Stokes equation.
Since the baseline method quickly leads to error blow-up,
we use another comparison criterion.
We define a stopping time
$t_K = \arg\max_t \frac{\norml
\mfu_t - \wht \mfu_t \normr_2}{\norml
\mfu_t \normr_2} \leq K$ where $K$ is an error threshold to be determined during the experiments.
We calculate the first
time the trajectory error reaches a threshold under different
Reynolds numbers and different mesh sizes.
The results are shown in~\cref{NS-table}.
Under almost all the scenarios, our method outperforms all the baseline algorithms by a larger $t_K$. This validates 
the effectiveness of our algorithm. Comparing \cref{NS-table} horizontally, one finds that the improvement of our
method is also increasing concerning the Reynolds number.
As the flow field generally becomes more complex (and sensitive)
with increasing Reynolds number,
our method is expected to bring bigger improvements.
This is consistent with our experiments.

% \begin{table}[ht]\label{NS-table}
% \centering
% \caption{Comparison of stopping time $t_K$ on Navier-Stokes equation over 10 experiments.}
% \begin{tabular}{|c|ccccc|} \hline
%      & $Re = 100$ & $Re = 200$ & $Re = 300$ & $Re = 400$ & $Re = 500$  \\
%     \hline
%     &  & & \textbf{grid size: }$128 \times 32$ &  &   \\
%     \hline
%     OLS  & $435 \pm 15$ & $301 \pm 20$ & $240 \pm 12$ & $177 \pm 17$ & $120 \pm 7$   \\
%     mOLS  & $420 \pm 12$ & $333 \pm 14$ & $254 \pm 12$ & $204 \pm 9$ & $210 \pm 7$   \\
%     aOLS  & $477 \pm 25$ & $392 \pm 31$ & $450 \pm 23$ & $239 \pm 25$ & $321 \pm 20$   \\
%     $\abbrvalgName$ &  $\mathbf{506 \pm 18}$ & $ \mathbf{497 \pm 28}$ & $\mathbf{482 \pm 29}$ & $\mathbf{442 \pm 21}$ & $\mathbf{452 \pm 13}$   \\
%     \hline
%     &  & & \textbf{grid size: }$256 \times 64$ &  &   \\
%     \hline
%     OLS  & $602 \pm 38$ & $571 \pm 35$ & $355 \pm 25$ & $301 \pm 11$ & $148 \pm 22$   \\
%     mOLS  & $623 \pm 18$ & $531 \pm 26$ & $489 \pm 20$ & $417 \pm 10$ & $298 \pm 9$   \\
%     aOLS  & $\mathbf{651 \pm 48}$ & $601 \pm 42$ & $\mathbf{595 \pm 39}$ & $378 \pm 31$ & $398 \pm 32$   \\
%     $\abbrvalgName$ &  $636 \pm 35$ & $\mathbf{602 \pm 28}$ & $585 \pm 23$ & $\mathbf{422 \pm 35}$ & $\mathbf{402 \pm 28}$ \\
%     \hline
% \end{tabular}
% \end{table}

\begin{table}[tb]
  \Large
  \caption{Comparison of stopping time $t_K$ on Navier-Stokes equation over 10 experiments.}
  \label{NS-table}
  \centering
  \resizebox{0.95\columnwidth}{!}{
  \begin{tabular}{p{2cm} p{3cm} p{3cm} p{3cm} p{3cm} p{3cm}}
  \toprule [1.5pt]
  \parbox{2cm}{  } &  $Re = 100$ & $Re = 200$ & $Re = 300$ & $Re = 400$ & $Re = 500$   \\ \midrule[1.5pt]
  & & \textbf{grid size:} & $128 \times 32$ & &  \\\midrule[0.5pt]
    OLS  & $435 \pm 15$ & $301 \pm 20$ & $240 \pm 12$ & $177 \pm 17$ & $120 \pm 7$   \\
    mOLS  & $420 \pm 12$ & $333 \pm 14$ & $254 \pm 12$ & $204 \pm 9$ & $210 \pm 7$   \\
    aOLS  & $477 \pm 25$ & $392 \pm 31$ & $450 \pm 23$ & $239 \pm 25$ & $321 \pm 20$   \\
    $\abbrvalgName$ &  $\mathbf{506 \pm 18}$ & $ \mathbf{497 \pm 28}$ & $\mathbf{482 \pm 29}$ & $\mathbf{442 \pm 21}$ & $\mathbf{452 \pm 13}$   \\  \midrule[0.5pt]
  & & \textbf{grid size:} & $256 \times 64$ & &  \\\midrule[0.5pt]
  OLS  & $602 \pm 38$ & $571 \pm 35$ & $355 \pm 25$ & $301 \pm 11$ & $148 \pm 22$   \\
    mOLS  & $623 \pm 18$ & $531 \pm 26$ & $489 \pm 20$ & $417 \pm 10$ & $298 \pm 9$   \\
    aOLS  & $\mathbf{651 \pm 48}$ & $601 \pm 42$ & $\mathbf{595 \pm 39}$ & $378 \pm 31$ & $398 \pm 32$   \\
    $\abbrvalgName$ &  $636 \pm 35$ & $\mathbf{602 \pm 28}$ & $585 \pm 23$ & $\mathbf{422 \pm 35}$ & $\mathbf{402 \pm 28}$
  \\\bottomrule[1.5pt]
  \end{tabular}
  }
\end{table}

\section{Conclusion}

In this paper, we establish a theoretical framework for machine-learning augmented
hybrid simulation problems, where a data-driven surrogate is used to accelerate
traditional simulation methods.
We identify the cause and effect of distribution
shift, i.e. the empirically observed phenomenon that
the simulated dynamics may be driven away from the support of the training data
due to systematic errors introduced by the data-driven surrogate and
magnified by the resolved components of the dynamics.
Based on this, we propose a tangent-space
regularized algorithm for training the surrogate for the unresolved part,
which incorporates the resolved model information to
control deviations from the true data manifold.
In the case of linear dynamics, we show our algorithm is
provably better than the traditional training method based on ordinary least
squares. Then, we validate our algorithm in numerical experiments, including
Turing instabilities in reaction-diffusion equations and fluid flow. In both cases,
our method outperforms baselines by better mitigating distribution shift,
thus reducing trajectory error. Important problems such as different
discretizations for training and simulating, 
and applications to turbulent flow simulations are beyond the scope of this
paper and left for future study.

\section*{Acknowledgements}

The research work presented is supported by the National Research Foundation, Singapore,
under the NRF fellowship (project No. NRF-NRFF13-2021-0005). We thank Zhixuan Li, Sohei Arisaka, 
Jiequn Han for valuable discussions.

\bibliography{main}
\bibliographystyle{siamplain}

\newpage
\appendix
\section{Linear dynamics}
\subsection{Auxiliary results}
We provide the proof of the following lemma.
\begin{lemma}\label{lem:matrix-exp-bound}
    Suppose $P$ is a matrix whose similarity transformation transforms $M$ to its Jordan form. The following bound for $M \in \mbR^{m\times m}$ holds
    \begin{equation*}
        \norml e^{Mt} \normr_2 \leq \cond(P)m^2(2 + t^{m-1})e^{\eig_{\max}(M) t}.
    \end{equation*}
\end{lemma}
\begin{proof}
  Writing the Jordan decomposition of the matrix $M$ as 
  \begin{equation}
      M = PJP^{-1},
  \end{equation}
  where $J$ is the Jordan normal form of $M$ with Jordan block $J_1, J_2, \cdots, J_k$. Then, the exponential can be written as
  \begin{equation}
      e^{Mt} = Pe^{Jt}P^{-1} = P\diag(e^{J_1t}, \ldots, e^{J_kt})P^{-1}.
  \end{equation}
  We can bound the matrix norm of the LHS by
  \begin{equation}
      \norml e^{Mt} \normr_2 \leq \cond(P)\norml e^{Jt} \normr_2 \leq \cond(P)\max_{i=1, \ldots, k} \norml e^{J_it} \normr_2.
  \end{equation}
  It suffices to estimate the norm of the exponential of a Jordan block. Suppose we have the following Jordan block of shape $m_i \times m_i$
  \begin{equation}
      J_i = \begin{pmatrix}
          \lambda_i & 1 & 0 & \cdots & 0 \\
          0 & \lambda_i & 1 & \cdots & 0 \\
          \vdots & \vdots & \vdots & \ddots & \vdots \\
          0 & 0 & 0 & \cdots & \lambda_i
      \end{pmatrix}, \quad e^{J_it} = e^{\lambda_i t}\begin{pmatrix}
        1 & t & t^2 & \cdots & \frac{t^{m_i-1}}{(m_i-1)!} \\
        0 & 1 & t & \cdots & \frac{t^{m_i-2}}{(m_i-2)!} \\
        \vdots & \vdots & \vdots & \ddots & \vdots \\
        0 & 0 & 0 & \cdots & 1
    \end{pmatrix}.
  \end{equation}
  Lastly, we can bound the norm of this exponential as
  \begin{equation}
      \norml e^{J_it} \normr_2 \leq e^{\Re(\lambda_i) t}\sum_{i, j=1}^{m_i} \lp e^{J_it} \rp_{ij} \leq m_i^2(1 + t^{m_i-1})e^{\Re(\lambda_i) t},
  \end{equation}
  which concludes our proof.
\end{proof}

\subsection{Comparison of different methods over linear dynamics}
Here, we compare our method with OLS in a toy linear dynamics in
\cref{fig:linear-cmp-supp}. As the OLS estimator does not involve any penalty term, we
further introduce the modified OLS (mOLS) with the following objective function
\begin{equation}\label{equ:mOLS-error}
    l_{mOLS}(\wht C) = \mbE\lp \norml (C^* - \wht C) \mfu\normr_2^2 + \lambda \norml \wht C \normr_F^2 \rp,
\end{equation}
to quantitatively compare the performance of estimators with different regularizations.

We fix the overall dynamics to be
\begin{equation}
    \mfu_{n+1} = (A+BC)\mfu_n = F\mfu_n, \quad F = \begin{pmatrix}
        0.95 & 0        \\
        0 & 1.2
    \end{pmatrix}.
\end{equation}
In each test case, $B$ is set to be the identity matrix and C is randomly generated with each element following i.i.d. uniform distribution over $[0, 1]$. The initial condition is sampled from $\mfu_0 
\sim ( \mcN(\mathbf{0},
1), 0)$, i.e. the first coordinate from i.i.d.
standard Gaussian and the second coordinates fixed to $0$. It is simple to
conclude that all the data will concentrate on the stable manifold $\mbR \times
\{0\}$. The observational noise scale is set to
$0.001$ and the number of time steps to $50$. Based on the trajectories simulated
from these initial conditions, we calculate the OLS and $\abbrvalgName$ estimators and their corresponding trajectory error.

The performance of the tangent-space regularized 
algorithm achieves better performance than the OLS algorithm and mOLS. Meanwhile, as we increase the penalty
$\lambda$, the error curves become closer to the best rate of the error rate inside the stable manifold. More precisely, we draw the error propagation rates in the stable and unstable manifold using 
two dashed lines. This result quantitatively supports our theory from three perspectives: 
\begin{itemize}
\item The trajectory error of the OLS simulator coincides with the rate over the unstable manifold, i.e. $e_1$ and exhibits exponential blow-up. This supports our error bound for 
$\mfu_{\mathrm{OLS}}.$
\item The trajectory error of the TR simulator has two stages. At the early stage, the error changes with a rate close to the error rate in the stable manifold, i.e. $e_2 < 0 < e_1$. At the late stage, 
the error accumulates at a rate similar to the rate in the unstable manifold and OLS error rate. This implies two components in the error bound for the TR algorithm where the first term dominates at the 
early stage while the other dominates late dynamics. Moreover, as $\lambda\rightarrow\infty$, the length of the first stage increases, meaning a smaller simulation error.
\item Lastly, our TR outperforms the OLS algorithm with $L^2$-regularization. Although the rates of error accumulation are similar, the initial error of our algorithm is much smaller than that of 
regularized OLS, as our estimator is unbiased while the mOLS estimator is biased.
\end{itemize}
    \begin{figure}[ht]\label{linear-cmp}
        \centering
        \centerline{\includegraphics[width=.8\linewidth]{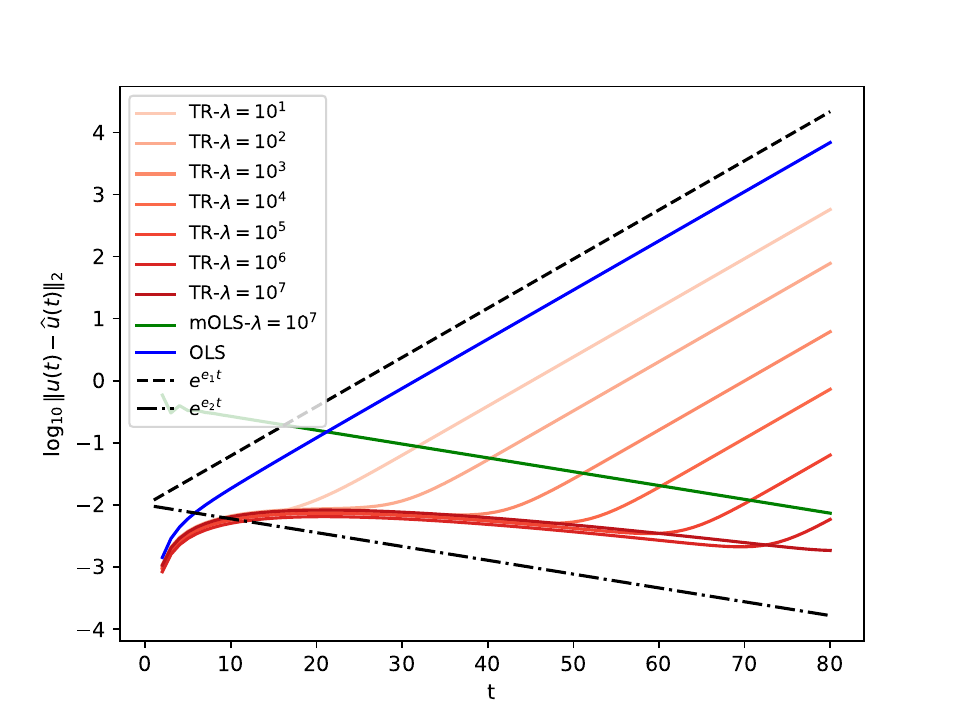}}
        \caption{We compare the simulation error of various estimators in a synthetic linear dynamic where
    the data subspace $V(V^{\perp})$ is the stable (unstable) manifold. This quantitatively supports our theory from three perspectives: 1). The trajectory error of the OLS simulator coincides with 
    the rate $e_1$ over the unstable manifold; 2). The trajectory error of our algorithm has two stages, i.e. $e_2$-rate at the beginning and $e_1$ later, which correspond exactly to two parts in our main 
    theorem. 3). As $\lambda\rightarrow\infty$, the length of the first stage increases, meaning a smaller simulation error.}
    \end{figure}

Moreover, We
elaborate on the relationship between the long time prediction accuracy with
$\lambda$ in \cref{fig:linear-cmp-supp}. The first observation is that both regularized
methods outperform the naive OLS method as regularization penalizes the deviation
along the unstable direction. Secondly, the error of mOLS and TR both decrease
as regularization becomes stronger. However, as $\lambda$ becomes larger, the
improvement of the mOLS algorithm ceases and cannot achieve better performance
than $10^{-3}$. This is because regularization $\norml \wht C
\normr_F$ does not distinguish the directions lying inside the manifold
containing the true trajectories (here the line parallel to $(1, 0)$) and
those pointing outward. Large $\lambda$ guarantees the components along unstable
directions are small and also forces the components along directions inside the manifold to be small, which leads to performance deterioration. However,
the proposed $\abbrvalgName$ estimator distinguishes between the directions lying inside and pointing outwards of the manifold and selectively penalizes the latter.
Therefore, the error made along the unstable direction is eliminated
as $\lambda \rightarrow \infty$ while at the same time, the prediction along the
tangent space is not affected. This guarantees a continuous improvement of the
trajectory prediction for the tangent-space regularized algorithm,
outperforming the classical norm-based regularization.
\begin{figure}[ht]
          \centering
          \label{fig:linear-cmp-supp}
    \centerline{\includegraphics[width=.7\linewidth]{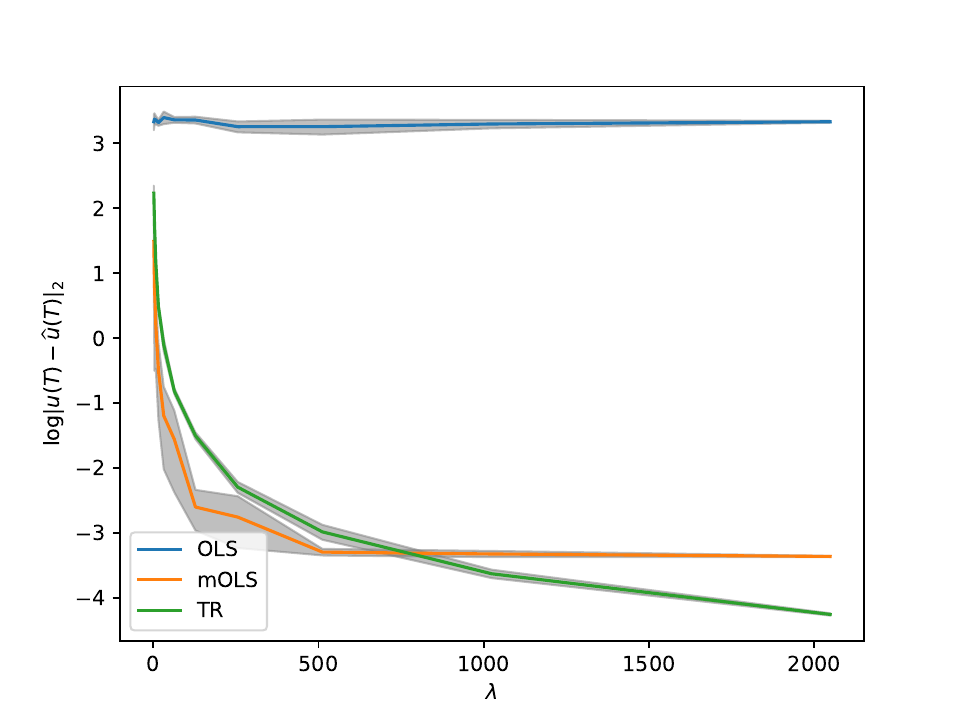}}
          \caption{
          In this figure, we compare our method with OLS and modified OLS. The
          horizontal axis denotes the strength of the penalty while the vertical
          axis denotes the accuracy of the prediction at terminal time. The
          scale of noise
          $\epsilon$ is set to
          $0.001$, time steps to $50$. All the cases have a low
          dimensional structure, i.e. only the first coordinate is non-zero. The shadow area represents the standard deviation calculated by 10 random tests.
          The first observation is that both regularized methods
          outperform the naive OLS method as regularization penalizes the
          deviation along the unstable direction. Secondly, the error of mOLS
          and TR both decrease as regularization becomes stronger. However, as
          $\lambda$ becomes greater, the improvement of the mOLS algorithm
          ceases and cannot achieve better performance than $10^{-3}$.}
\end{figure}

\section{Numerical experiments details}
\label{numeric-setup}
In this section, we first describe how to prepare the training data of reaction-diffusion equations and Navier-Stokes equations. Then, we discuss the experimental details of the neural networks 
and algorithms.
\subsection{Training data of reaction-diffusion equation}
We present the details of the training data preparation for the reaction-diffusion equation 
\begin{equation}
    \begin{aligned}
        	\frac{\p \mfu}{\p t} & = D \Delta \mfu + \phi(\mfu), \quad T \in [0, 20], 	\\
		\phi(\mfu) & = \phi(u, v) = \begin{pmatrix}
			u - u^3 - v + \alpha	\\
			\beta(u - v)
		\end{pmatrix}.
    \end{aligned}
\end{equation}
where $\mfu = (u(x, y, t), v(x, y, t))^T \in \mbR^2$ are two interactive
components, $D$ is the diffusion matrix, and $\mfR(\mfu)$ is source
term for the reaction. The parameters are set to be $\alpha=0.01,\beta=1.0,D=\begin{pmatrix}
    \gamma & 0    \\
    0 & 2\gamma
\end{pmatrix}, \gamma \in \{0.05, 0.10, 0.15, 0.20, 0.25\}$ in the simulation. To obtain the initial condition $\mfu_0$, we first sample i.i.d. from a normal
distribution:
\bequ
	u(x, y, 0), v(x, y, 0) \sim \mcN(0, 1).
\eequ
Then, we solve the equation from this initial condition for $200$ time steps with $dt = 0.01$ to obtain a physical velocity profile $\mfu$, which is taken to be the initial condition. The reason to 
discard the first $200$ time step velocity profiles is that they appear to be less physical and we believe adding them into the training set will deteriorate the performance of the data-driven model. 
Lastly, the boundary condition is taken to be periodic in the simulation domain
$[0, 6.4]^2$. Given all this information, we solve the equation using the Crank-Nicolson scheme where the non-linear term is treated explicitly. The $\Delta t$ of the simulation is set to $0.01$ and 
a total time step $1000$. To obtain the coarse-fine grid correction term, we first simulate the equation on the grid of size $2n$ from initial condition $\mfu_0^{2n}$, which provides velocity 
profiles $\mfu_i^{2n}, i=0,1,2,\cdots, T$. Then, for each time step $k$, starting from the velocity profile $\mfu_k^{2n}$, we first restrict it via $R_{2n}^{n}$ to obtain a velocity profile 
$\mfu_k^{n}$ over grid $n$. Next, we iterate this profile one step by Crank-Nicolson scheme and interpolate it back to grid $2n$ via $I_n^{2n}$. The procedure is summarized as
\bequ
    \mfy_k^{2n} = \mfu_{k+1}^{2n} - I_n^{2n} \circ f_n(R_{2n}^{n}(\mfu_k^{2n})).
\eequ
Along with $\mfu_k^{2n}$, $(\mfu_k^{2n}, \mfy_k^{2n})$ becomes a data pair we used to learn the unresolved model.
\subsection{Training data of Navier-Stokes equation}
We provide the details on the numerical scheme used to solve the Navier-Stokes equation with special focus on the resolved and unresolved parts of the solver.
This is also the structure of our MLHS algorithm. This equation is written as follows
\begin{equation}
    \begin{aligned}
        	\frac{\p \mfu}{\p t} + (\mfu \cdot \nabla)\mfu -  \nu \Delta \mfu & =   \nabla p, \quad T \in [0, 1], 	\\
		\nabla \cdot \mfu & = 0,
    \end{aligned}
\end{equation}
where $\mfu = (u(x, y, t), v(x, y, t))^T \in \mbR^2$ is velocity and $p$
pressure.

The computational domain of the Navier-Stokes equation is $[0, 4]\times[0, 1]$ and we use staggered grid pf size $128\times 32$ and $256 \times 64$. The boundary condition is given by
\begin{equation}
    \begin{aligned}
        u\Big|_{y=0} = v\Big|_{y=0} & = u\Big|_{y=1} = v\Big|_{y=1} = 0,      \\
        p\Big|_{x=4} = 0, \frac{\p v}{\p n}\Big|_{x=4} = 0, u\Big|_{x=0} & = \exp\{-50(y-y_0)^2\}, \\
        v\Big|_{x=0} & = \sin t \cdot \exp\{-50(y-y_0)^2\}.
    \end{aligned}
\end{equation}
While the initial value of velocity $u, v$ is set to vanish except for the boundary $x=0$.
To enforce the boundary condition, we use the ghost-cell method and the true computational boundary lies in the middle of the first grid cell. The grid size during the numerical simulation is given 
by $u: (n_x+2) \times (n_y+2), v: (n_x+2) \times (n_y+1), p: n_x \times n_y$ where $(n_x, n_y) = (128, 32), (256, 64).$ The pressure is obtained via solving a Poisson equation during 
each time step. Moreover, since the boundary ghost cells are only of numerical importance to force the boundary condition, we discard them when building the dataset of the Navier-Stokes simulation, 
i.e. the dataset consists of tuples $(u, v, p), u, v, p \in \mbR^{n_x \times n_y}$. Moreover, as we use staggered grid for simulation, the collocation points for $u, v, p$ are different for all the grid cells. We linearly interpolate the value of $u, v$ at the boundaries of cells to the middle point of cells which coincides with the collocation points for pressure to generate a consistent dataset. Similar to the reaction-diffusion case, we collect the field data after the first $100$ iterations to avoid 
those unphysical fluid field configurations.

\subsection{Different distribution shift extents with different simulation parameters}
For the Navier-Stokes equation,
the degree of distribution shift is induced by varying
the Reynolds number, which is verified in~\cref{fig:NSRe-ds}. It shows that as the Reynolds number is increased, both
the trajectory simulation error and distribution shift increase
for almost all sampled time steps.
\begin{figure}[ht]\label{fig:NSRe-ds}
          \centering
          \centerline{\includegraphics[width=.8\linewidth]{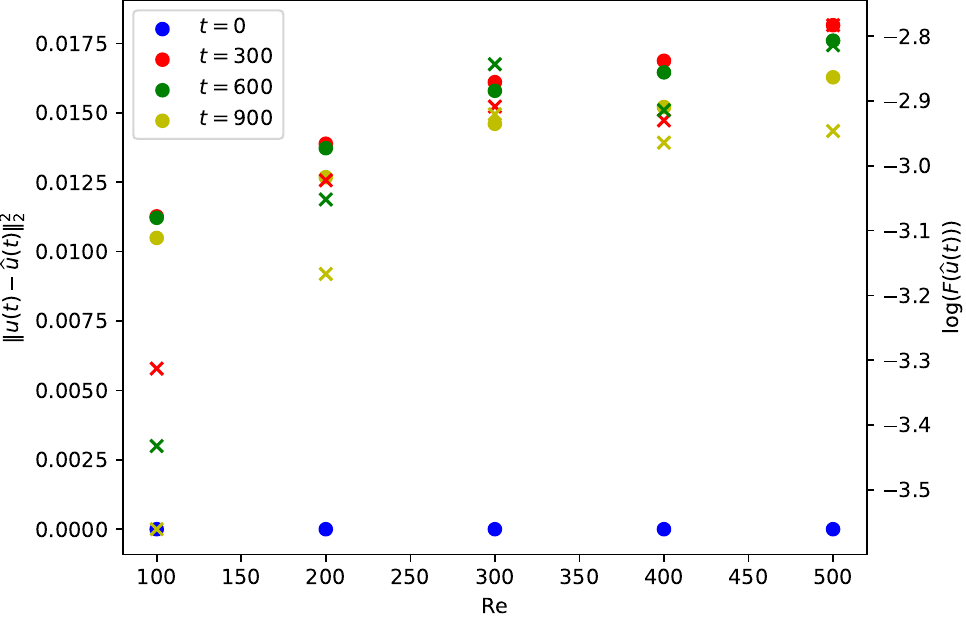}}
          \caption{
          In this figure, we illustrate the distribution shift and error of the
          simulated trajectories under different Reynolds numbers. The
          solid dots represent the error at given time steps and simulation
          parameter Re and the `$\times$'s denote the log of the
          corresponding distribution shift calculated using autoencoder.
          Different time steps are distinguished using different colors.
          From the results, we conclude that as the Reynolds number
          increases, both the simulated error and distribution shift increase
          for almost all the sampled time steps.
          This verifies the intuition that the Navier-Stokes equation with a greater Reynolds number
          suffers from more severe distribution shift.
          }
\end{figure}

\subsection{Network model details}
To capture the fluid patterns precisely, we use a modification of the
U-net\cite{ronneberger2015u}. The main structure is an autoencoder whose encoder and decoder steps are constructed as convolutional blocks which serve as a non-linear dimension reduction. The skip 
connection is used between
parallel layer to retain the information. The detailed architectures are described below where we only focus on one specific grid size.

We use the convolutional autoencoder to identify the intrinsic manifold of our training data. Specifically, the encoder is composed of a sequence of convolutional layers followed by max-pooling that 
extracts the information from a gradually large scale. The decoder is attached to the output of the encoder and is composed of a sequence of convolutional layers followed by upsampling to recover the 
state variable. For the depth of the encoder and decoder as well as the number of channels in each layer, we use the same setting as the network model for surrogate modeling discussed in the next paragraph. 
The key difference between the convolutional encoder we used for characterizing the underlying data manifold structure and the UNet for surrogate modeling is that UNet has skip connections between the 
layers of the same size in the encoder and decoder while the autoencoder does not. This is to avoid trivial reconstruction of the intrinsic manifold. For a more detailed discussion on using the convolutional autoencoder to approximate the nonlinear trial manifold, we recommend 
reading \cite{lee2020model}. Regarding the training issue of the autoencoder, we use all the data (no train-test split). The depth and size of the bottleneck layer of the network need to be tuned to make sure the 
training loss is small enough. Our empirical observation is that if the optimization error of the autoencoder is not small enough, e.g. $10^{-1}$, the performance of the regularized estimator will not 
show great improvement over the OLS estimator, meaning that the regularization is not effective. If the error is trained below $10^{-3}$, then our algorithm has significant improvement over 
the benchmark and the performance is no longer sensitive to further decrease the training loss of the autoencoder.

For the reaction-diffusion equation, the domain is a square so we use square block with size $128, 64, 32, 16, 8, 4, 8, 16, 32, 64, 128$ and number of blocks $2, 4, 8, 16, 16, 32$, $16, 16, 8, 4, 2$ 
respectively. In the downsampling stage where the block size shrinks to half each time, we use a convolutional layer with kernel size $4 \times 4$, padding $1$, and stride $2$ followed by a 
batch normalization layer. In the upsampling stage, between each block where the size doubles, we first apply an upsampling layer with scale factor $2$ followed by a convolutional layer with 
kernel size $3$, padding $1$, and stride $1$ and again a batch-normalization layer. For the Navier-Stokes equation, the domain is a rectangle and the blocks are modified to adjust the 
height-width ratio. The block size follows $256\times 64, 128\times 32, 128\times 32, 64\times 16, 32\times 8, 16 \times 4, 8 \times 2, 16 \times 4, 32\times 8, 64\times 16, 128\times 32, 256\times 64$ 
with the same block number at each level as reaction-diffusion cases. The only difference is that the output layer only contains one block since we are predicting the pressure, which only has one component.

During training different network models, we used Adam optimizer~\cite{kingma2014adam} with PyTorch~\cite{paszke2019pytorch} scheduling ``ReduceLROnPlateau'' and initial learning rate $10^{-4}$. For 
each specific scenario, i.e. Navier-Stokes equation with Reynolds number $400$ and grid size $128\times 32$, we generate 10 trajectories using high fidelity numerical solver. During training, one of 
them is randomly chosen as the test trajectory and others are used as training data. The training epoch of the autoencoder is set to $3000$ and of the OLS and the $\abbrvalgName$ estimator is set to $5000$. 
The batch size is chosen to be $1000$, which equals the trajectory length. The constant $K$ in $t_K = \arg\max_t \norml \wht \mfu_t - \mfu_t \normr \leq K$ is taken to be $100$ during the experiment.
\subsection{Choice of regularization strength}\label{lambda}
From a theoretical perspective, $\lambda \rightarrow \infty$ will provide the best estimator that stabilizes the hybrid simulation. However, as $\lambda$ increases, it takes more time steps 
for the objective function to reach a fixed threshold. In practice, as we increase the value of $\lambda$ while fixing the time steps, the final loss function keeps increasing and the estimator 
error, i.e. the former part of the loss function also increases. We compare four different choices of $\lambda = 1, 10, 100, 1000$ and fix it to $10$ throughout the whole experiment.

\end{document}